\documentclass[12pt]{article}

\usepackage[latin1]{inputenc}
\usepackage[T1]{fontenc}
\usepackage{kpfonts}
\usepackage[english]{babel}
\usepackage{xspace}
\usepackage{eurosym} 
\usepackage{enumerate} 
\usepackage{verbatim} 
\usepackage{cancel}
\usepackage{changepage}
\usepackage[pdftex,leftbars,xcolor]{changebar}
\usepackage{multicol} 

\usepackage{multirow}
\usepackage{tabularx}

\allowdisplaybreaks[2]

\usepackage[pdftex,a4paper,centering]{geometry}
\usepackage{pdflscape}

\usepackage{fancyhdr}

\usepackage{amsmath,amsthm}
\usepackage{amssymb,latexsym,amsfonts}

\usepackage{aliascnt}
\numberwithin{equation}{section}

\usepackage{mathtools}
\usepackage{relsize,exscale}
\usepackage{rotating}

\usepackage{etoolbox}

\usepackage{etex}

\usepackage[all]{xy}

\usepackage{xcolor}

\pdfoutput=1
\usepackage[pdftex,%
bookmarks=true,%
bookmarksnumbered=true,%
plainpages=false,%
breaklinks=true,%
colorlinks=false,%
urlcolor=blue,%
citecolor=green,%
linkcolor=red,%
pdftitle={Brackets with (tau,sigma) derivations},%
pdfauthor={Olivier Elchinger}]{hyperref}
\usepackage{doi}

\usepackage[hyperpageref]{backref}
\backreffrench
\renewcommand*{\backref}[1]{}  
\renewcommand*{\backrefalt}[4]{
  \ifcase #1 %
  \relax
  \or
(Cited page~#2.)%
  \else
(Cited pages~#2.)%
  \fi}

\newtheorem{theorem}{Theorem}[section]

\newaliascnt{cor}{theorem}
\newtheorem{cor}[cor]{Corollary}
\aliascntresetthe{cor}

\newaliascnt{prop}{theorem}
\newtheorem{prop}[prop]{Proposition}
\aliascntresetthe{prop}

\newaliascnt{lemma}{theorem}
\newtheorem{lemma}[lemma]{Lemma}
\aliascntresetthe{lemma}

\theoremstyle{definition}
\newaliascnt{defi}{theorem}
\newtheorem{defi}[defi]{Definition}
\aliascntresetthe{defi}

\newaliascnt{example}{theorem}
\newtheorem{example}[example]{Example}
\aliascntresetthe{example}

\theoremstyle{remark}
\newaliascnt{remark}{theorem}
\newtheorem{remark}[remark]{Remark}
\aliascntresetthe{remark}

\addto\extrasenglish{%
}

\newcommand{\ie}{\textit{i.e.} \/}

\newcommand{\numberset}[1]{\mathbb{#1}}

\newcommand{\nat}{\numberset{N}}
\newcommand{\intg}{\numberset{Z}}

\newcommand{\complex}{\numberset{C}}

\newcommand{\A}{\mathcal{A}}

\DeclareMathOperator{\pr}{\textsf{pr}}

\DeclareMathOperator{\Ann}{Ann}

\usepackage{authblk}

\title{Brackets with $(\tau,\sigma)$-derivations and $(p,q)$-deformations of Witt and Virasoro algebras}
\author[1]{Olivier Elchinger}
\author[2]{Karl Lundengård}
\author[1]{Abdenacer Makhlouf}
\author[2]{Sergei D.~Silvestrov}

\affil[1]{
Laboratoire de Math\'{e}matiques, Informatique  et Applications\\

Universit\'{e} de Haute Alsace, Mulhouse, France
}
\affil[2]{
Mathematics and Applied Mathematics\\

Mälardalens Högskola, Västerås, Sweden
}

\begin{document}

\maketitle


\begin{abstract}
The aim of this paper is to  study some brackets defined on $(\tau,\sigma)$-derivations satisfying quasi-Lie identities. Moreover, we provide examples of $(p,q)$-deformations of  Witt and  Virasoro algebras as well as  $\mathfrak{sl}(2)$ algebra. These constructions generalize the results obtained by Hartwig, Larsson and Silvestrov on $\sigma$-derivations, arising in connection with  discretizations and deformations of algebras of vector fields. 

\end{abstract}



\section*{Introduction}

Hom-Lie algebras and related algebraic structures have recently become rather popular, due in part to the prospect of having a general framework in which one can produce many types of natural deformations of (Lie) algebras, in particular $q$-deformations which are of interest to both mathematicians and physicists. 

The area of quantum deformations (or $q$-deformations) of Lie algebras began a period of rapid expansion around 1985 when Drinfeld and Jimbo  independently considered deformations of $U(\mathfrak{g})$, the universal enveloping algebra of a semisimple Lie algebra $\mathfrak{g}$, motivated, among other things, by their applications to quantum integrable systems, the Yang-Baxter equation and quantum inverse scattering methods. This method developed by Faddeev and his
collaborators gave rise to a great interest in quantum groups and Hopf algebras. Since then
several other versions of ($q$-)deformed Lie algebras have appeared, especially in physical contexts
such as string theory. The main objects for these deformations were infinite-dimensional
algebras, primarily the Heisenberg algebras (oscillator algebras) and the Virasoro algebra, see
\cite{AlGau89,AiSa91,CEP90,CILPP91,CKL90,CPP90,CZ90,DamKul91,Das92,Hu99,K92,Liu92} and the references therein. The main tools are skew-derivations or $\sigma$-derivations, which are generalized derivations twisting the Leibniz rule by means of a linear map.

In a series of papers \cite{HLS06,LS05,LS07} the authors have developed a new quasi-deformation scheme leading from Lie algebras to a broader class of quasi-Lie algebras and subclasses of quasi-Hom-Lie algebras and Hom-Lie algebras. These
classes of algebras are tailored in a way suitable for simultaneous treatment of the Lie algebras, Lie superalgebras, the color Lie algebras and the deformations arising in connection with twisted, discretized or deformed derivatives and corresponding generalizations, discretizations and deformations of vector fields and differential calculus. Indeed, in \cite{HLS06,LS05,LS05b,LS07}, it has been shown that the class of quasi-Hom-Lie algebras contains as a subclass on the one hand the color Lie algebras and in particular Lie superalgebras and Lie
algebras, and on the other hand various known and new single and multi-parameter families of algebras obtained using twisted derivations and constituting deformations and quasi-deformations of universal enveloping algebras of Lie and color Lie algebras and of algebras of vector-fields.

Quasi-Lie algebras are nonassociative algebras for which the skew-symmetry and the Jacobi identity are twisted by several deforming twisting maps and also the Jacobi identity in quasi-Lie and quasi-Hom-Lie algebras in general contains 6 twisted triple bracket terms. In the subclass of Hom-Lie algebras skew-symmetry is untwisted. Thus Hom-Lie algebras, in the contrast with quasi-Hom-Lie algebras, do not contain super or color Lie algebras that are not Lie algebras. The Jacobi identity in Hom-Lie algebras is however twisted by a single linear map and contains 3 terms as in Lie algebras and color Lie algebras (6 terms in quasi-Lie Jacobi identity are combined pairwise).

This kind of algebraic structures introduced by Hartwig, Larsson and Silvestrov were later further extended to other type of nonassociative algebras (associative, Lie-admissible, Leibniz, Jordan, Alternative, Malcev, $n$-ary algebras \ldots) as well as  their dual and graded versions, see \cite{AM10,BM14,CG11,EM14,FGS09,Gohr10,Mak:Almeria,MS09,MP14,She12,Yau10,Yau-homology,Yau09b,Yau11,AAM14,AMS14,MY14,AMM11,AMS10}.
The main feature of this type of algebras, called Hom-algebras, is that the usual  identities are twisted by one or several deforming twisting maps, which lead to many
interesting results. \\

Our purpose in this paper is to construct internal brackets based $(\tau,\sigma)$-derivations leading to quasi-Lie algebras. This work  generalizes the constructions  in \cite{HLS06} which deals with $\sigma$-derivations, and were used to obtain $q$-deformations of  Witt and Virasoro algebras. We recover the results on  $\sigma$-derivations  when $\tau = id$. Moreover, we provide examples of $(p,q)$-deformations of Witt, Virasoro and $\mathfrak{sl}(2)$ algebras.

In \autoref{Sec:preliminaries}, we give the definition and some properties of $(\tau,\sigma)$-derivations, in particular the fact that they form a free rank-one module when $\tau \neq \sigma$. We also recall the definitions of quasi-Lie algebras, Hom-Lie algebras and their extensions. In \autoref{Sec:Bracket-st-der}, we define quasi-Lie and Hom-Lie algebras (\autoref{Thm:Bracket-Gen-Der} and \autoref{Thm:Forced-Bracket}), extending the main theorem of \cite{HLS06} to the case of $(\tau,\sigma)$-derivations. We provide in  \autoref{Sec:Examples} applications of  these theorems  to obtain $(p,q)$-deformations of Witt, Virasoro and $\mathfrak{sl}(2)$ algebras. We study the relations between these deformations and $q$-deformations obtained in \cite{HLS06,LS07}. In  \autoref{Sec:summarized-diagrams}, we  summarize the various deformations for Witt and $\mathfrak{sl}(2)$ algebras and their relationships using diagrams. Finally, in \autoref{Sec:another-ex}, we provide other examples of internal brackets based  on $(\tau,\sigma)$-derivations.

\section{Preliminaries} \label{Sec:preliminaries}

In this first preliminary section, we recall the definitions and some properties of $(\tau,\sigma)$-derivations and Hom-Lie algebras. To lighten the computations, the composition of maps symbol $\circ$ is sometimes omitted.

\subsection{\texorpdfstring{Generalities on $(\tau,\sigma)$-derivations}{Generalities on (tau,sigma)-derivations}}

Let $R$ be a ring and $\A$ be an associative $R$-algebra. We consider two algebra endomorphisms $\tau$ and $\sigma$ of $A$.

\begin{defi}
A \emph{$(\tau,\sigma)$-derivation} $D : \A \to \A$ is an $R$-linear map such that the following generalized Leibniz identity
\begin{equation}
D(ab)=D(a)\tau(b)+\sigma(a)D(b)
\end{equation}
is satisfied, with $a,b \in \A$. The set of all $(\tau,\sigma)$-derivation on $\A$ is denoted by $\mathcal{D}_{\tau,\sigma}(\A)$.
\end{defi}

\begin{remark}
Suppose that the algebra $\A$ is commutative. Then a $(\tau,\sigma)$-derivation is also a $(\sigma,\tau)$-derivation because
\[
D(ab) = D(ba) = D(b)\tau(a) + \sigma(b)D(a) = D(a)\sigma(b) + \tau(a)D(b).
\]
A $(id,\sigma)$-derivation (or $(\sigma,id)$-derivation if the algebra $A$ is commutative) will simply be called \emph{$\sigma$-derivation}. We recover classical derivations when $\sigma = \tau = id$.
\end{remark}

\begin{prop}
Let $\A$ be an associative $R$-algebra and $\tau,\tau',\sigma,\sigma'$ algebra endomorphisms of $\A$. Let $D$ be a $(\tau,\sigma)$-derivation and $D'$ a $(\tau',\sigma')$-derivation. If $\tau$ and $\tau'$ commute, $\sigma$ and $\sigma'$ commute, and $D$ commutes with $\tau'$ and $\sigma'$, and $D'$ commutes with $\tau$ and $\sigma$, then $[D,D'] \coloneqq D \circ D' - D' \circ D$ is a $(\tau \tau',\sigma \sigma')$-derivation.
\end{prop}

\begin{proof}
For $a,b \in \A$, we have
\begin{gather*}
\begin{aligned}
(D \circ D')(ab) &= D(D'(a)\tau'(b) + \sigma'(a)D'(b)) \\
&= DD'(a)\tau\tau'(b) + \sigma D'(a) D \tau'(b) + D \sigma'(a) \tau D'(b) + \sigma\sigma'(a)DD'(b)
\end{aligned}
\intertext{and}
\begin{aligned}
(D' \circ D)(ab) &= D'(D(a)\tau(b) + \sigma(a)D(b)) \\
&= D'D(a)\tau'\tau(b) + \sigma' D(a) D' \tau(b) + D' \sigma(a) \tau' D(b) + \sigma'\sigma(a)D'D(b),
\end{aligned}
\end{gather*}
substracting the two expressions, the four middle-terms cancel by using the hypothesis to give
\[
[D,D'](ab) = [D,D'](a) \tau\tau'(b) + \sigma\sigma'(a)[D,D'](b)
\]
which shows that $[D,D'] = D \circ D' - D' \circ D$ is a $(\tau \tau',\sigma \sigma')$-derivation.
\end{proof}

Denoting by
\[
\mathcal{D}(\A) = \bigcup_{\tau,\sigma\ \in\ Z(\mathcal{L}_R(\A))} \mathcal{D}_{\tau,\sigma}(\A) \quad \subset \ \mathcal{L}_R(\A)
\]
the set of all $(\tau,\sigma)$-derivations, for all endomorphisms $\tau,\sigma$ that commute with each other and with all linear maps of $\A$, the preceding proposition gives a structure of Lie algebra on this subset of commuting $R$-linear maps on $\A$.

\begin{cor}
Endowed with the bracket $[D,D'] \coloneqq  D \circ D' - D' \circ D$, the subspace $\mathcal{D}(\A)$ is a Lie algebra.
\end{cor}

We fix now for the next sections $\tau$ and $\sigma$ endomorphisms of $\A$, and will work only on $(\tau,\sigma)$-derivations $\mathcal{D}_{\tau,\sigma}(\A)$.
\begin{lemma} \label{Lem:sigma-tau-difference}
Let $\A$ be an associative $R$-algebra and $\tau,\sigma$ two different algebra endomorphisms of $\A$. For an element $c$ in the center of $\A$, define an $R$-linear map $D : \A \to \A$ by
\begin{equation}
D(f) = c (\tau(f) - \sigma(f)).
\end{equation}
Then $D$ is a $(\tau,\sigma)$-derivation.
\end{lemma}

\begin{proof}
We have
\begin{align*}
D(fg) &= c (\tau(fg) - \sigma(fg)) = c (\tau(f) - \sigma(f)) \tau(g) + \sigma(f) (\tau(g) - \sigma(g))c \\
& = D(f)\tau(g) + \sigma(f) D(g).
\end{align*}
\end{proof}

\begin{prop}
Let $\A$ be an associative $R$-algebra and $\tau,\sigma$ two different algebra endomorphisms of $\A$. Let $\alpha$ be an algebra endomorphism of $\A$ and $D$ a $(\tau,\sigma)$-derivation. Then $\alpha \circ D$ is a $(\alpha \circ \tau,\alpha \circ \sigma)$-derivation and $D \circ \alpha$ is a $(\tau \circ \alpha,\sigma \circ \alpha)$-derivation.
\end{prop}

\begin{proof}
For $a,b \in \A$, we have
\[
(\alpha \circ D)(ab) = \alpha(D(a)\tau(b)+\sigma(a)D(b)) = (\alpha \circ D)(a) (\alpha \circ \tau)(b) + (\alpha \circ \sigma)(a) (\alpha \circ D)(b).
\]
So $\alpha \circ D$ is a $(\alpha \circ \tau,\alpha \circ \sigma)$-derivation ; same argument the other way around.
\end{proof}

Notice that if $\alpha : \A \to \A$ is invertible, then $\alpha^{-1} : \A \to \A$ is also an algebra morphism.

\begin{cor} \hfill \label{Cor:One-to-Two_Der}
\begin{enumerate}[(i)]
\item Let $D$ be a $\sigma$-derivation. If $\tau$ is an algebra endomorphism of $\A$, then $\tau \circ D$ is a $(\tau,\tau \circ \sigma)$-derivation.
\item Let $D$ be a $(\tau,\sigma)$-derivation. \\
If $\tau$ is invertible, then 
$\tau^{-1} \circ D$ is a $(\tau^{-1} \circ \sigma)$-derivation. \\
If $\sigma$ is invertible, then 
$\sigma^{-1} \circ D$ is a $(\sigma^{-1} \circ \tau,id)$-derivation, and if $\A$ is commutative, it is also a $(\sigma^{-1} \circ \tau)$-derivation.
\end{enumerate}
\end{cor}

\bigskip

When $\sigma(x)a = a\sigma(x)$ (or $\tau(x)a = a\tau(x)$) for all $x, a \in \A$ and in particular when $\A$ is commutative, $\mathcal{D}_{\tau,\sigma}(\A)$ carries a natural left (or right) $\A$-module
structure defined by $(a,D) \mapsto a \cdot D : x \mapsto aD(x)$. For clarity we will always denote the module multiplication by $\cdot$ and the algebra multiplication in $\A$ by juxtaposition. We will sometimes extend maps $\alpha : \A \to \A$ on $\mathcal{D}_{\tau,\sigma}(\A)$ by $\overline{\alpha}(a \cdot D) = \alpha(a) \cdot D$. \\

If $a, b \in \A$ we shall write $a~|~b$ if there is an element $c \in \A$ such that $ac = b$. If $S \subseteq \A$ is a subset of $\A$, a \emph{greatest common divisor} of $S$, noted $\gcd(S)$, is defined as an element of $\A$ satisfying, for all $a \in S$
\begin{equation*}
\gcd(S)~|~a \qquad \text{and for}\ b \in \A,\quad  b~|~a \Rightarrow b~|~\gcd(S)
\end{equation*}
It follows directly from the definition that
\begin{equation} \label{gcd-subsets}
S \subseteq T \subseteq A \Rightarrow \gcd(T)~|~\gcd(S)
\end{equation}
whenever $\gcd(S)$ and $\gcd(T)$ exist. If $\A$ is a unique factorization domain one can show that a $\gcd(S)$ exists for any nonempty subset $S$ of $\A$ and that this element is unique up to a multiple of an invertible element in $\A$. Thus we are allowed to speak of \emph{the} $\gcd$.

\begin{theorem}[\cite{HLS06}, Theorem 4 p.319] \label{Thm:Rank-One-Module}
Let $\tau$ and $\sigma$ be different algebra endomorphisms on a unique factorization domain $\A$. Then $\mathcal{D}_{\tau,\sigma}(A)$ is free of rank one as an $\A$-module with generator
\begin{equation}
\Delta \coloneqq \frac{\tau - \sigma}{g} : x \mapsto \frac{(\tau - \sigma)(x)}{g},
\end{equation}
where $g = \gcd((\tau - \sigma)(\A))$.
\end{theorem}

\begin{defi}
The \emph{annihilator} $\Ann(D)$ of an element $D \in \mathcal{D}_{\tau,\sigma}(\A)$ is
\[
\Ann(D) = \{a \in \A,\ a \cdot D = 0\}.
\]
\end{defi}
It is easy to see that $\Ann(D)$ is an ideal of $\A$. Notice that for an algebra automorphism $\alpha$, we have $\Ann(\alpha \circ D) = \alpha(\Ann(D))$.

For $D \in \mathcal{D}_{\tau,\sigma}(\A)$, we denote by
\[
\A \cdot D = \{a \cdot D,\ a \in \A\}
\]
the cyclic $\A$-submodule generated by $D$.

\begin{example}
Various examples of $(\tau,\sigma)$-derivations are given in \cite[Table 1, p.11]{Hartwig-master02}. We list in \autoref{Table:sigma-tau-der-examples} such examples and some others, mainly obtained using \autoref{Lem:sigma-tau-difference}. On a algebra of functions, they include the usual derivative, the shift difference and different Jackson derivatives.

\begin{table}
\hspace*{-1.2em}
\begin{tabular}{c|c|c|c}
\hline
Operator $D$ & $D(f(t))$ & $D((fg)(t))$ & $(\tau,\sigma)$ \\
\hline
Differentiation & $f'(t)$ & $D(f(t)) g(t) + f(t) D(g(t))$ & $(id,id)$ \\
Shift $S$ & $f(t+1)$ & $f(t+1) D(g(t))$ & $(S,0)$ \\
Shift difference & $f(t+1)-f(t)$ & $D(f(t)) g(t) + f(t+1) D(g(t))$ & $(S,id)$ \\
$q$-dilatation $T_q$ & $f(qt)$ & $f(qt) D(g(t))$ & $(T_q,0)$ \\[1ex]
Jackson $q$-derivative & $\dfrac{f(t)-f(qt)}{t-qt}$ & $D(f(t)) g(t) + f(qt) D(g(t))$ & $(id,T_q)$ \\[2ex]
Jackson symmetric $q$-derivative & $\dfrac{f(q^{-1}t)-f(qt)}{q^{-1}t-qt}$ & $D(f(t)) g(q^{-1}t) + f(qt) D(g(t))$ & $(T_{q^{-1}},T_q)$ \\[2ex]
Jackson $(p,q)$-derivative & $\dfrac{f(pt)-f(qt)}{pt-qt}$ & $D(f(t)) g(pt) + f(qt) D(g(t))$ & $(T_p,T_q)$ \\[2ex]
$p$-dilatation derivative & $f'(pt)$ & $D(f(t)) g(pt) + f(pt) D(g(t))$ & $(T_p,T_p)$
\end{tabular}
\caption{Some $(\tau,\sigma)$-derivations and their Leibniz rules}
\label{Table:sigma-tau-der-examples}
\end{table}
\end{example}

\subsection{Quasi-Lie algebras, Hom-Lie algebras and extensions}

Hom-Lie algebras were first introduced by Hartwig, Larsson and Silvestrov in \cite{HLS06} precisely by using $\sigma$-derivations. We will provide similar construction using $(\tau,\sigma)$-derivations. Since then, other Hom-structure such as Hom-associative algebras, Hom-Poisson algebras, \textit{etc.}, were studied, see \cite{BEM12,MakSil-struct,HomAlgHomCoalg,MakSil-def,AEM11,Yau-Poisson}.
We recall here the present definition and some properties of Hom-Lie algebras.

\begin{defi}
A Hom-Lie algebra is a triple $(\A,[~,~],\alpha)$ where $\A$ is a $R$-algebra, $[~,~] : \A \otimes \A \to \A$ is an $R$-bilinear skew-symmetric bracket and $\alpha : \A \to A$ is a linear map, called \emph{twist}, such that for all $x,y,z \in \A$
\begin{equation} \label{Hom-Jacobi}
\circlearrowleft_{x,y,z}{[\alpha(x),[y,z]]} = 0.
\end{equation}
The equation \eqref{Hom-Jacobi} is called Hom-Jacobi identity.
\end{defi}

If $\alpha = id$, we recover the definition of a Lie algebra, so Lie algebras will be considered as Hom-Lie algebras with the identity as twist map. \\

Let $\A = \left(\A,\mu,\alpha \right)$ and $\A^{\prime} = \left(\A^{\prime},\mu^{\prime}, \alpha^{\prime}\right)$ be two Hom-Lie algebras. A linear map $\varphi \ : \A \rightarrow \A^{\prime}$ is a \emph{morphism of Hom-Lie algebras} if%
\[
\mu^{\prime} \circ (\varphi \otimes \varphi) = \varphi \circ \mu \quad \text{and} \qquad \varphi \circ \alpha = \alpha^{\prime}\circ \varphi.
\]
It is said to be a \emph{weak morphism} if holds only the first condition. \\

More generally, the notion of quasi-Lie algebra was discussed in \cite{LS05,HomAlgHomCoalg}. Let $\mathcal{L}_R(\A)$ be the set of $R$-linear maps on $\A$.
\begin{defi} (Larsson, Silvestrov \cite{LS05}) \label{def:quasiLiealg}
A \emph{quasi-Lie algebra} is a tuple \mbox{$(\A,[~,~],\alpha,\beta,\omega,\theta)$}
where
\begin{itemize}
    \item $\A$ is a $R$-algebra,
    \item $[~,~]: \A \otimes \A \to \A$ is a bilinear map called a \emph{product} or a \emph{bracket} on $\A$,
    \item $\alpha,\beta : \A \to \A$, are $R$-linear maps,
    \item $\omega : D_\omega \to \mathcal{L}_R(V)$ and $\theta : D_\theta \to \mathcal{L}_R(V)$ is a map with domain of definition \mbox{$D_\omega, D_\theta \subseteq \A \otimes \A$},
\end{itemize}
such that the following conditions hold:
\begin{description}
    \item[$\omega$-symmetry] The product satisfies a generalized skew-symmetry condition
\[
[x,y] = \omega(x,y) [y,x], \quad \text{ for all } (x,y) \in D_\omega;
\]
    \item[quasi-Jacobi identity] The bracket satisfies a generalized Jacobi identity
\[
\circlearrowleft_{x,y,z} \Big\{\,\theta(z,x) \Big([\alpha(x),[y,z]] + \beta[x,[y,z]] \Big)\Big\}=0, \quad \text{ for all } (z,x),(x,y),(y,z) \in D_\theta.
\]
\end{description}
\end{defi}
This class includes (Hom-)Lie superalgebras, $\intg$-graded (Hom-)Lie algebras, and color (Hom-)Lie algebras (\cite{AM10}). Specifying in the definition of quasi-Lie algebras $D_\omega = \A \otimes \A$, $\beta=0_\A$ and $\theta(x,y) = \omega(x,y) = -id_\A$ for all $(x,y) \in D_\omega = D_\theta$, we recover the Hom-Lie algebras with twisting linear map $\alpha$.

\bigskip

The following result shows how to obtain other Hom-Lie algebras from a given one. It can be used to twist classical Lie algebras into Hom-Lie algebras using a morphism. It also works with other types of structures (see \cite{BEM12} for example), so it is sometimes referred as the `twisting principle'.
\begin{theorem}[{\cite[Theorem 3.2]{Yau-Poisson}}] \label{Thm:TwistAssLie}
Let $(\A, [~,~],\alpha)$ be a Hom-Lie algebra and $\rho : \A \to \A$ be a weak morphism, i.e. $\rho \circ [~,~] = [~,~] \circ \rho^{\otimes 2}$. Then $\A_\rho = (\A, [~,~]_\rho = \rho \circ [~,~], \alpha_\rho = \rho \circ \alpha)$ is a Hom-Lie algebra.
\end{theorem}

\begin{proof}
For $x,y,z \in \A$, we have $[y,x]_\rho = -[x,y]_\rho$ and
\begin{equation*}
\begin{split}
& \circlearrowleft_{x,y,z} [(\rho \circ \alpha)(x),[y,z]_\rho]_\rho \\
&= \circlearrowleft_{x,y,z} \rho([\rho(\alpha(x)),\rho([y,z])]) \\
&= \rho^2\left(\circlearrowleft_{x,y,z}[\alpha(x),[y,z]]\right) \\
&= 0
\end{split}
\end{equation*}
since $[~,~]$ satisfies the Hom-Jacobi condition with $\alpha$ as twist map, so $\A_\rho$ is a Hom-Lie algebra.
\end{proof}
In particular, if $\rho : \A \to \A$ is a Hom-Lie algebra morphism, then it is also a morphism of the Hom-Lie algebra $\A_\rho$.

We will say that the Hom-Lie algebra $\A_\rho = (\A, [~,~]_\rho = \rho \circ [~,~], \alpha_\rho = \rho \circ \alpha)$ is a \emph{twist} of the Hom-Lie algebra $(\A,[~,~],\alpha)$, or that the two Hom-Lie algebras are \emph{twist-equivalent}.

\begin{prop}
Let $(\A,\mu,\alpha)$ be a Hom-Lie algebra, and $\rho : \A \to \A$ a weak morphism. Suppose that
\[
\varphi : (A,\mu,\alpha) \to (\A, \rho \circ \mu, \rho \circ \alpha)
\]
is an isomorphism of Hom-Lie algebras between $\A$ and its $\rho$-twist $A_\rho$. Then $\rho$ measures the default of commutativity between $\varphi$ and $\alpha$, that is $\rho = \varphi \circ \alpha \circ \varphi^{-1} \circ \alpha^{-1}$.
\end{prop}

\begin{proof}
Since $\varphi : \A \to \A_\rho$ is a morphism, we have
\begin{gather*}
\rho \circ \mu \circ (\varphi \otimes \varphi) = \varphi \circ \mu \\
\varphi \circ \alpha = \rho \circ \alpha \circ \varphi.
\end{gather*}
The second equation is equivalent to $\rho = \varphi \circ \alpha \circ \varphi^{-1} \circ \alpha^{-1}$. Using that $\rho$ is a weak morphism, the first equation also gives $\mu \circ (\rho \circ \varphi \otimes \rho \circ \varphi) = \varphi \circ \mu$.
\end{proof}

\begin{cor}
If a Lie algebra $(\A,\mu,id)$ is isomorphic to a twist-equivalent Hom-Lie algebra $(\A, \rho \circ \mu, \rho)$, then the twist $\rho$ is trivial, $\rho = id$.
\end{cor}

\begin{remark}
Given a Lie algebra, this shows that twists of it are never isomorphic to it (as Hom-Lie algebras). This is also true for associative algebras.
\end{remark}

\bigskip

We recall here the definitions and results on central extensions of Hom-Lie algebras developed in \cite{HLS06}. The only minor difference is that the homomorphism $\zeta$ in the old definition is now replaced by the twist map $\alpha$. The link between these two definitions is $\alpha = id + \zeta$, but one should convince oneself that the results are the same by replacing $(L,\zeta)$ by $(L,\alpha)$ Hom-Lie algebras.

\begin{defi}
An extension of a Hom-Lie algebra $(L, \alpha)$ by an abelian Hom-Lie algebra $(\mathfrak{a}, \alpha_\mathfrak{a})$ is a commutative diagram with exact rows
\begin{equation}
\xymatrix{
0 \ar[r] & \mathfrak{a} \ar[r]^{\imath} \ar[d]_{\alpha_\mathfrak{a}} & \hat{L} \ar[r]^{\pr} \ar[d]_{\hat{\alpha}} & L \ar[r] \ar[d]_{\alpha} & 0  \\
0 \ar[r] & \mathfrak{a} \ar[r]^{\imath} & \hat{L} \ar[r]^{\pr} & L \ar[r] & 0
}
\end{equation}
where $(\hat{L},\hat{\alpha})$ is a Hom-Lie algebra. We say that the extension is \emph{central} if
\begin{equation}
\imath(\mathfrak{a}) \subseteq Z(\hat{L}) \coloneqq \{x \in \hat{L},\ [x,\hat{L}]_{\hat{L}} = 0 \}.
\end{equation}
\end{defi}

\begin{theorem}[\cite{HLS06}, Theorem 21 p.335] \label{Thm:ExtHomLie-uniqueness}
Suppose $(L,\alpha)$ and $(\mathfrak{a},\alpha_\mathfrak{a})$ are Hom-Lie algebras with $\mathfrak{a}$ abelian. If there exists a central extension $(\hat{L},\hat{\alpha})$ of $(L,\alpha)$ by $(\mathfrak{a},\alpha_\mathfrak{a})$ then for every section $s : L \to \hat{L}$ there is a $g_s \in \bigwedge^2(L,\mathfrak{a})$ and a linear map $f_s : \hat{L} \to \mathfrak{a}$ such that
\begin{gather}
f_s \circ \imath = \alpha_\mathfrak{a} \\
g_s(\alpha(x),\alpha(y)) = f_s([s(x),s(y)]_{\hat{L}})
\intertext{and}
\circlearrowleft_{x,y,z} g_s(\alpha(x),[y,z]_L) = 0 \label{indep-2cocycle}
\end{gather}
for all $x,y,z \in L$. Moreover, equation \eqref{indep-2cocycle} is independent of the choice of section $s$.
\end{theorem}

\begin{defi}
A central Hom-Lie algebra extension $(\hat{L},\hat{\alpha})$ of $(L,\alpha)$ by $(\mathfrak{a},\alpha_\mathfrak{a})$ is called trivial if there exists a linear section $s : L \to \hat{L}$ such that
\begin{equation}
g_s(x,y) = 0
\end{equation}
for all $x,y \in L$.
\end{defi}

\begin{theorem}[\cite{HLS06}, Theorem 24 p.336] \label{Thm:ExtHomLie-existence}
Suppose $(L,\alpha)$ and $(\mathfrak{a},\alpha_\mathfrak{a})$ are Hom-Lie algebras with $\mathfrak{a}$ abelian. Then for every $g \in \bigwedge^2(L,\mathfrak{a})$ and every linear map $f : L\oplus \mathfrak{a} \to \mathfrak{a}$ such that
\begin{gather}
f(0,a) = \alpha_\mathfrak{a}(a) \quad \text{for $a \in \mathfrak{a}$}, \\
g(\alpha(x),\alpha(y)) = f([x,y]_L,g(x,y))
\intertext{and}
\circlearrowleft_{x,y,z} g(\alpha(x),[y,z]_L) = 0
\end{gather}
for $x,y,z \in L$, there exists a Hom-Lie algebra $(\hat{L},\hat{\alpha})$ which is a central extension of $(L,\alpha)$ by $(\mathfrak{a},\alpha_\mathfrak{a})$.
\end{theorem}

\section{\texorpdfstring{Brackets on $(\tau,\sigma)$-derivations}{Brackets on (tau,sigma)-derivations}} \label{Sec:Bracket-st-der}

We let $\A$ be a commutative associative unital algebra and fix morphisms $\tau,\sigma : \A \to \A$ and an element $\Delta$ of $\mathcal{D}_{\tau,\sigma}(\A)$. The goal is to define a bracket on $(\tau,\sigma)$-derivations which satisfies some quasi-Jacobi identity; and when $\tau = id$, to recover the bracket defined in \cite{HLS06} for $\sigma$-derivations:
\begin{equation}
[a \cdot \Delta,b \cdot \Delta]_\sigma \coloneqq  (\sigma(a) \cdot \Delta) \circ (b \cdot \Delta) - (\sigma(b) \cdot \Delta) \circ (a \cdot \Delta).
\end{equation}

\subsection{\texorpdfstring{With $\tau$ invertible}{With tau invertible}}

If $\tau$ is invertible, then $\tau^{-1} \circ \Delta$ is a $(\tau^{-1} \circ \sigma)$-derivation (\autoref{Cor:One-to-Two_Der}). We rewrite the previous bracket in this situation, and apply $\tau$ on both sides, using that $\tau^{-1}$ is a morphism:
\[
\tau([a \cdot \tau^{-1} \circ \Delta,b \cdot \tau^{-1} \circ \Delta]_{\tau^{-1} \circ \sigma}) = (\sigma(a) \cdot \Delta) \circ (b \cdot \tau^{-1} \circ \Delta) - (\sigma(b) \cdot \Delta) \circ (a \cdot \tau^{-1} \circ \Delta) :=: [\tau(a) \cdot \Delta,\tau(b) \cdot \Delta]_{\tau,\sigma}
\]
The last `equality' is the way we want to define our new bracket on $(\tau,\sigma)$-derivations, having in mind that $\tau$ goes through the bracket in the left expression. We have the following theorem, which introduces an $R$-algebra structure on $\A \cdot \Delta$.

\begin{theorem} \label{Thm:Bracket-Gen-Der}
Let $\A$ be a commutative associative unital algebra and $\tau,\sigma : \A \to \A$ different morphisms with $\tau$ invertible. Let $\Delta$ be a $(\tau,\sigma)$-derivation.
If the equation
\begin{equation} \label{annulation-condition}
(\sigma \circ \tau^{-1})(\Ann(\Delta)) \subseteq \Ann(\Delta)
\end{equation}
holds, then the map $[~,~]_{\tau,\sigma} : \A \cdot \Delta \times \A \cdot \Delta \to \A \cdot \Delta$ defined by setting
\begin{equation} \label{bracket-def}
[a \cdot \Delta,b \cdot \Delta]_{\tau,\sigma} \coloneqq ((\sigma \circ \tau^{-1})(a) \cdot \Delta) \circ (\tau^{-1}(b) \cdot \tau^{-1} \circ \Delta) - ((\sigma \circ \tau^{-1})(b) \cdot \Delta) \circ (\tau^{-1}(a) \cdot \tau^{-1} \circ \Delta)
\end{equation}
for $a,b \in \A$ is a well-defined $R$-algebra bracket on the $R$-linear space $\A \cdot \Delta$, and it satisfies the following identities for $a,b,c \in \A$:
\begin{gather}
[a \cdot \Delta,b \cdot \Delta]_{\tau,\sigma} = \Big( (\sigma \circ \tau^{-1})(a) (\Delta \circ \tau^{-1})(b) - (\sigma \circ \tau^{-1})(b) (\Delta \circ \tau^{-1})(a) \Big) \cdot \Delta  \label{interior-bracket} \\
[b \cdot \Delta,a \cdot \Delta]_{\tau,\sigma} = - [a \cdot \Delta,b \cdot \Delta]_{\tau,\sigma}. \label{skew-symmetry}
\intertext{Moreover if there exists an element $\delta \in \A$ such that}
\Delta \circ \tau^{-1} \circ \sigma \circ \tau^{-1} = \delta \cdot (\sigma \circ \tau^{-1} \circ \Delta \circ \tau^{-1}), \label{commutation-condition}
\intertext{then we have}
\circlearrowleft_{a,b,c} \Big( [\sigma(\tau^{-1}(a)) \cdot \Delta,[b \cdot \Delta,c \cdot \Delta]_{\tau,\sigma}]_{\tau,\sigma} + \delta \cdot [a \cdot \Delta,[b \cdot \Delta,c \cdot \Delta]_{\tau,\sigma}]_{\tau,\sigma} \Big) = 0. \label{quasi-Jacobi-identity}
\end{gather}
\end{theorem}
We will refer to equation \eqref{quasi-Jacobi-identity} as quasi-Jacobi identity. \\

This theorem gives a quasi-Lie algebra structure on $\A \cdot \Delta$, with $\alpha = \sigma \circ \tau^{-1}$, $\beta = \delta$, $\theta = id_\A$ and $\omega = - id_\A$.

\begin{remark}
The bracket defined in the theorem depends on the maps $\tau$ and $\sigma$, as specified explicitly in the notation, but also depends on the generator $\Delta$ of the cyclic submodule $\A \cdot \Delta$ of $\mathcal{D}_{\tau,\sigma}(\A)$ and should be written $[~,~]_{\Delta,\tau,\sigma}$. Suppose that another generator $\Delta'$ of $\A \cdot \Delta$ is choosen, with $\Delta' = u \Delta,\ u \in \A$, then, like for the case of $\sigma$-derivations,  the `base-change' relation is given by
\[
\sigma(\tau^{-1}(u)) [a \cdot \Delta',b \cdot \Delta']_{\Delta',\tau,\sigma} = u [a \cdot \Delta',b \cdot \Delta']_{\Delta,\tau,\sigma}.
\]
The generator is fixed for the most part of this paper, so we suppress the dependence on the generator from the bracket notation. If $\A$ has no zero-divisors, then the result \cite[Proposition 9 p.15]{HLS06} that the dependence of the generator $\Delta$ is not essential still holds and is proven with similar arguments.

Furthermore, the bracket $[~,~]_{\tau,\sigma}$ will simply be written $[~,~]$ to lighten the computations, and the symbol of composition $\circ$ will often be omitted.
\end{remark}

\begin{proof}
We first show that $[~,~]$ is a well-defined function, that is, if $a_1 \cdot \Delta = a_2 \cdot \Delta$, then
\[
[a_1 \cdot \Delta,b \cdot \Delta] = [a_2 \cdot \Delta,b \cdot \Delta] \qquad \text{and} \qquad [b \cdot \Delta,a_1 \cdot \Delta] = [b \cdot \Delta,a_2 \cdot \Delta].
\]
Now $a_1 \cdot \Delta = a_2 \cdot \Delta$ is equivalent to $a_1 - a_2 \in \Ann(\Delta)$. Therefore, $\tau^{-1}(a_1 - a_2) \in \tau^{-1}(\Ann(\Delta)) = \Ann(\tau^{-1} \circ \Delta)$ and using \eqref{annulation-condition}, we also have $\sigma(\tau^{-1}(a_1 - a_2)) \in \Ann(\Delta)$. Hence
\begin{align*}
& [a_1 \cdot \Delta,b \cdot \Delta] - [a_2 \cdot \Delta,b \cdot \Delta] = \\
={}& (\sigma\tau^{-1}(a_1) \cdot \Delta) \circ (\tau^{-1}(b) \cdot \tau^{-1}\Delta) - (\sigma\tau^{-1}(b) \cdot \Delta) \circ (\tau^{-1}(a_1) \cdot \tau^{-1}\Delta) \\
&{}- (\sigma\tau^{-1}(a_2) \cdot \Delta) \circ (\tau^{-1}(b) \cdot \tau^{-1}\Delta) + (\sigma\tau^{-1}(b) \cdot \Delta) \circ (\tau^{-1}(a_2) \cdot \tau^{-1}\Delta) \\
={}& (\sigma\tau^{-1}(a_1 - a_2) \cdot \Delta) \circ (\tau^{-1}(b) \cdot \tau^{-1}\Delta) - (\sigma\tau^{-1}(b) \cdot \Delta) \circ (\tau^{-1}(a_1 - a_2) \cdot \tau^{-1}\Delta) = 0,
\end{align*}
and a similar computation for the second equality.

Next we prove \eqref{interior-bracket}, which also shows that $\A \cdot \Delta$ is closed under the bracket $[~,~]$. Let $a,b,c \in \A$, then, since $\Delta$ is a $(\tau,\sigma)$-derivation on $\A$ we have
\begin{align*}
& [a \cdot \Delta,b \cdot \Delta](c) = (\sigma\tau^{-1}(a)\cdot \Delta)\Big((\tau^{-1}(b) \cdot \tau^{-1}\Delta)(c)\Big) - (\sigma\tau^{-1}(b)\cdot \Delta)\Big((\tau^{-1}(a) \cdot \tau^{-1}\Delta)(c)\Big) \\
&= \sigma\tau^{-1}(a)\Delta\big(\tau^{-1}(b)\tau^{-1}\Delta(c)\big) - \sigma\tau^{-1}(b)\Delta\big(\tau^{-1}(a)\tau^{-1}\Delta(c)\big) \\
&= \sigma\tau^{-1}(a) \Big( \Delta\tau^{-1}(b)\Delta(c) + \sigma\tau^{-1}(b)\Delta\tau^{-1}(c) \Big) - \sigma\tau^{-1}(b) \Big( \Delta\tau^{-1}(a)\Delta(c) + \sigma\tau^{-1}(a)\Delta\tau^{-1}(c) \Big) \\
&= \big( \sigma\tau^{-1}(a)\Delta\tau^{-1}(b) - \sigma\tau^{-1}(b)\Delta\tau^{-1}(a) \big)\Delta(c) + \big( \sigma\tau^{-1}(a)\sigma\tau^{-1}(b) - \sigma\tau^{-1}(b)\sigma\tau^{-1}(a) \big) \Delta\tau^{-1}\Delta(c).
\end{align*}
Since $\A$ is commutative, the last term is zero, thus \eqref{interior-bracket} is true. The skew-symmetry identity \eqref{skew-symmetry} is clear from the definition \eqref{bracket-def}. Using the linearity of $\tau,\sigma,\Delta$ and the definition of $[~,~]$ or the formula \eqref{interior-bracket}, it is also easy to see that $[~,~]$ is bilinear. \\

It remains to prove \eqref{quasi-Jacobi-identity}. Using \eqref{interior-bracket} and that $\Delta$ is a $(\tau,\sigma)$-derivation on $\A$ we get
\begin{align*}
& [\sigma\tau^{-1}(a) \cdot \Delta,[b \cdot \Delta,c \cdot \Delta]] = [\sigma\tau^{-1}(a) \cdot \Delta,\Big( \sigma\tau^{-1}(b)\Delta\tau^{-1}(c) - \sigma\tau^{-1}(c)\Delta\tau^{-1}(b) \Big) \cdot \Delta] \\
={} & \bigg\{ (\sigma\tau^{-1})^2(a)\Delta\Big( \tau^{-1}\sigma\tau^{-1}(b)\tau^{-1}\Delta\tau^{-1}(c) - \tau^{-1}\sigma\tau^{-1}(c)\tau^{-1}\Delta\tau^{-1}(b) \Big) \\
& \qquad - \sigma\tau^{-1}\Big( \sigma\tau^{-1}(b)\Delta\tau^{-1}(c) - \sigma\tau^{-1}(c)\Delta\tau^{-1}(b) \Big) \Delta\tau^{-1}\sigma\tau^{-1}(a) \bigg\} \cdot \Delta \\
={} & \bigg\{ (\sigma\tau^{-1})^2(a) \Big( \Delta\tau^{-1}\sigma\tau^{-1}(b) \Delta\tau^{-1}(c) + (\sigma\tau^{-1})^2(b)(\Delta\tau^{-1})^2(c) \\
& \hspace*{6em} - \Delta\tau^{-1}\sigma\tau^{-1}(c)\Delta\tau^{-1}(b) - (\sigma\tau^{-1})^2(c)(\Delta\tau^{-1})^2(b) \Big) \\
& \quad - (\sigma\tau^{-1})^2(b)\sigma\tau^{-1}\Delta\tau^{-1}(c)\Delta\tau^{-1}\sigma\tau^{-1}(a) + (\sigma\tau^{-1})^2(c)\sigma\tau^{-1}\Delta\tau^{-1}(b)\Delta\tau^{-1}\sigma\tau^{-1}(a) \bigg\} \cdot \Delta \\
={} & \bigg\{ (\sigma\tau^{-1})^2(a)\Delta\tau^{-1}\sigma\tau^{-1}(b)\Delta\tau^{-1}(c) - (\sigma\tau^{-1})^2(a)\Delta\tau^{-1}\sigma\tau^{-1}(c)\Delta\tau^{-1}(b) \\
& \quad + (\sigma\tau^{-1})^2(a)(\sigma\tau^{-1})^2(b)(\Delta\tau^{-1})^2(c) - (\sigma\tau^{-1})^2(a)(\sigma\tau^{-1})^2(c)(\Delta\tau^{-1})^2(b) \\
& \quad - (\sigma\tau^{-1})^2(b)\sigma\tau^{-1}\Delta\tau^{-1}(c)\Delta\tau^{-1}\sigma\tau^{-1}(a) + (\sigma\tau^{-1})^2(c)\sigma\tau^{-1}\Delta\tau^{-1}(b)\Delta\tau^{-1}\sigma\tau^{-1}(a) \bigg\} \cdot \Delta.
\end{align*}
Summing cyclically on $a,b,c$, the second line of the last expression vanishes, and also the third line using \eqref{commutation-condition}, so we have
\begin{equation} \label{quasi-Jacobi-identity-first-part}
\begin{split}
& \circlearrowleft_{a,b,c} [\sigma\tau^{-1}(a) \cdot \Delta,[b \cdot \Delta,c \cdot \Delta]] = \\
={} & \circlearrowleft_{a,b,c} \Big\{ (\sigma\tau^{-1})^2(a)\Delta\tau^{-1}\sigma\tau^{-1}(b)\Delta\tau^{-1}(c) - (\sigma\tau^{-1})^2(a)\Delta\tau^{-1}\sigma\tau^{-1}(c)\Delta\tau^{-1}(b) \Big\} \cdot \Delta
\end{split}
\end{equation}
We now consider the second term of \eqref{quasi-Jacobi-identity}. First note that from \eqref{interior-bracket} we have
\[
[b \cdot \Delta,c \cdot \Delta] = \Big( \Delta\tau^{-1}(c)\sigma\tau^{-1}(b) - \Delta\tau^{-1}(b)\sigma\tau^{-1}(c) \Big) \cdot \Delta
\]
since $\A$ is commutative. Using this and \eqref{interior-bracket} we get
\begin{align*}
& \delta \cdot [a \cdot \Delta,[b \cdot \Delta,c \cdot \Delta]] = \delta \cdot [a \cdot \Delta,\Big( \Delta\tau^{-1}(c)\sigma\tau^{-1}(b) - \Delta\tau^{-1}(b)\sigma\tau^{-1}(c) \Big) \cdot \Delta] \\
={} & \delta \cdot \bigg\{ \sigma\tau^{-1}(a) \Delta\Big( \tau^{-1}\Delta\tau^{-1}(c)\tau^{-1}\sigma\tau^{-1}(b) - \tau^{-1}\Delta\tau^{-1}(b)\tau^{-1}\sigma\tau^{-1}(c) \Big) \\
& \qquad - \sigma\tau^{-1} \Big( \Delta\tau^{-1}(c)\sigma\tau^{-1}(b) - \Delta\tau^{-1}(b)\sigma\tau^{-1}(c) \Big) \Delta\tau^{-1}(a) \bigg\} \cdot \Delta \\
={} & \delta \cdot \bigg\{ \sigma\tau^{-1}(a) \Big( (\Delta\tau^{-1})^2(c)\sigma\tau^{-1}(b) + \sigma\tau^{-1}\Delta\tau^{-1}(c)\Delta\tau^{-1}\sigma\tau^{-1}(b) \\
& \hspace*{6em} - (\Delta\tau^{-1})^2(b)\sigma\tau^{-1}(c) - \sigma\tau^{-1}\Delta\tau^{-1}(b) \Delta\tau^{-1}\sigma\tau^{-1}(c) \Big) \\
& \qquad - \sigma\tau^{-1}\Delta\tau^{-1}(c)(\sigma\tau^{-1})^2(b)\Delta\tau^{-1}(a) + \sigma\tau^{-1}\Delta\tau^{-1}(b)(\sigma\tau^{-1})^2(c)\Delta\tau^{-1}(a) \bigg\} \cdot \Delta \\
={} & \bigg\{ \delta \sigma\tau^{-1}(a)(\Delta\tau^{-1})^2(c)\sigma\tau^{-1}(b) - \delta \sigma\tau^{-1}(a)(\Delta\tau^{-1})^2(b)\sigma\tau^{-1}(c) \\
& \quad + \delta \sigma\tau^{-1}(a)\sigma\tau^{-1}\Delta\tau^{-1}(c) \Delta\tau^{-1}\sigma\tau^{-1}(b) - \delta \sigma\tau^{-1}(a)\sigma\tau^{-1}\Delta\tau^{-1}(b) \Delta\tau^{-1}\sigma\tau^{-1}(c) \\
& \quad - \delta \sigma\tau^{-1}\Delta\tau^{-1}(c)(\sigma\tau^{-1})^2(b)\Delta\tau^{-1}(a) + \delta \sigma\tau^{-1}\Delta\tau^{-1}(b)(\sigma\tau^{-1})^2(c)\Delta\tau^{-1}(a) \bigg\} \cdot \Delta
\end{align*}
In this last expression, the second line vanishes using \eqref{commutation-condition} and commutativity, and summing cyclically on $a,b,c$, the first line also vanishes, so we have
\begin{equation}
\begin{split}
& \circlearrowleft_{a,b,c} \delta \cdot [a \cdot \Delta,[b \cdot \Delta,c \cdot \Delta]] = \\
={} & \circlearrowleft_{a,b,c} \Big\{ - \delta \sigma\tau^{-1}\Delta\tau^{-1}(c)(\sigma\tau^{-1})^2(b)\Delta\tau^{-1}(a) + \delta \sigma\tau^{-1}\Delta\tau^{-1}(b)(\sigma\tau^{-1})^2(c)\Delta\tau^{-1}(a) \Big\} \cdot \Delta
\end{split}
\end{equation}
which, using \eqref{commutation-condition}, is exactly the opposite of \eqref{quasi-Jacobi-identity-first-part}. We have thus proved the quasi-Jacobi identity \eqref{quasi-Jacobi-identity}.
\end{proof}

\begin{cor}
Under the same hypothesis, if $\sigma \tau^{-1}(\delta) = \delta$ and $\Delta(\delta) = 0$ (for example when $\delta = \lambda 1_\A$ is a multiple of the unit of $\A$), then the previous six-term quasi-Jacobi identity reduces to a Hom-Jacobi identity and $\A \cdot \Delta$ is a Hom-Lie algebra with linear twist map $\alpha = \sigma \tau^{-1} + \delta id_\A$.
\end{cor}

\begin{proof}
Using the same kind of computation than in the previous proof to get \eqref{quasi-Jacobi-identity}, we obtain
\begin{align*}
[\delta a \cdot \Delta,[b \cdot \Delta,c \cdot \Delta]] ={}& \sigma \tau^{-1}(\delta) \cdot [a \cdot \Delta,[b \cdot \Delta,c \cdot \Delta]] \\
& - \Big( (\sigma \tau^{-1})^2(b) \sigma \tau^{-1} \Delta \tau^{-1}(c) - (\sigma \tau^{-1})^2(b) \sigma \tau^{-1} \Delta \tau^{-1}(c) \Big) \Delta(\delta) \tau(a).
\end{align*}
\end{proof}

When the algebra $\A$ is an unique factorization domain, there is automatically an element $\delta$ which satisfies the relation of `commutation' \eqref{commutation-condition}.

\begin{prop} \label{Prop:Bracket-Gen-Der-UFD}
If $\A$ is an unique factorization domain and $\tau,\sigma : \A \to \A$ are two different algebra endomorphisms with $\tau$ invertible, then the equation \eqref{commutation-condition},
\[
\forall\ a \in \A, \qquad \Delta(\tau^{-1}(\sigma(\tau^{-1}(a)))) = \delta \sigma(\tau^{-1}(\Delta(\tau^{-1}(a))))
\]
holds with
\begin{equation}
\delta = \frac{\sigma(\tau^{-1}(g))}{g},
\end{equation}
where $g = \gcd((\tau-\sigma)(\A))$.
\end{prop}

\begin{proof}
In the case of an unique factorization domain $\A$ and $\tau,\sigma : \A \to \A$ two different algebra endomorphisms with $\tau$ invertible, by \autoref{Thm:Rank-One-Module} we have that
\[
\mathcal{D}_{\tau,\sigma}(\A) = \A \cdot \Delta
\]
where $\Delta = \dfrac{\tau - \sigma}{g}$ and $g = \gcd((\tau - \sigma)(\A))$. Let $y \in \A$ and set
\[
x \coloneqq \frac{\tau-\sigma}{g}(\tau^{-1}(y)) = \frac{y-\sigma(\tau^{-1}(y))}{g}.
\]
Then we have
\[
\sigma\tau^{-1}(g) \sigma\tau^{-1}(x) = \sigma(\tau^{-1}gx) = \sigma\tau^{-1}(\tau - \sigma)(\tau^{-1}(y)) = (\tau - \sigma)\tau^{-1}\sigma(\tau^{-1}(y)) = \sigma\tau^{-1}(y) - (\sigma\tau^{-1})^2(y).
\]
Since $g = \gcd((\tau-\sigma)(\A))$ and $(\tau-\sigma)(\A) = (id-\sigma\tau^{-1})(\A)$, using \eqref{gcd-subsets}, we have that $g$ is also a $\gcd$ of $(id-\sigma\tau^{-1})(\A)$ so $g~|~(id-\sigma\tau^{-1})(g) = g - \sigma(\tau^{-1}(g))$ and hence $g~|~\sigma\tau^{-1}(g)$. Then, dividing by $g$ and substitute the expression for $x$ we obtain
\[
\frac{\sigma\tau^{-1}(g)}{g} \sigma\tau^{-1}\left(\frac{\tau-\sigma}{g}\right)(\tau^{-1}(y)) = \frac{\tau-\sigma}{g}\tau^{-1}\sigma(\tau^{-1}(y))
\]
or, with $\Delta = \dfrac{\tau - \sigma}{g}$, this rewrites
\( \displaystyle
\frac{\sigma\tau^{-1}(g)}{g} \sigma\tau^{-1}\Delta\tau^{-1}(y) = \Delta\tau^{-1}\sigma\tau^{-1}(y).
\)
This shows that \eqref{commutation-condition} holds with
\[
\delta = \sigma\tau^{-1}(g)/g.
\]
\end{proof}

Since $\A$ has no zero-divisors and $\sigma \neq \tau$, it follows that $\Ann(\Delta) = 0$, and so \eqref{annulation-condition} is clearly true. Hence we can use \autoref{Thm:Bracket-Gen-Der} to define an algebra structure on $\mathcal{D}_{\tau,\sigma}(\A) = \A \cdot \Delta$ which satisfies \eqref{skew-symmetry} and \eqref{quasi-Jacobi-identity} with $\delta = \sigma\tau^{-1}(g)/g$.

\begin{remark} \label{Rmk:BaseChangeRelation}
The choice of the greatest common divisor is ambiguous. (We can choose any associated element, that is, the greatest common divisor is only unique up to a multiple by an invertible element). So this $\delta$ can be replaced by any $\delta' = u\delta$ where $u$ is a unit (that is, an invertible element). The derivation $\Delta$ is then replaced by an associated $\Delta'$. If $g'$ is another greatest common divisor related to $g$ by $g' = ug$, then
\[
\delta' = \frac{\sigma\tau^{-1}(ug)}{ug} = \frac{\sigma\tau^{-1}(u) \sigma\tau^{-1}(g)}{ug} = \frac{\sigma\tau^{-1}(u)}{u} \delta
\]
and $\sigma\tau^{-1}(u)/u$ is clearly a unit since $u$ is a unit. Therefore \eqref{commutation-condition} changes because
\[
\Delta' = \frac{\Delta}{u} = \frac{\tau - \sigma}{ug} \quad \Rightarrow \quad \Delta'\tau^{-1} \sigma\tau^{-1} = \frac{\sigma\tau^{-1}(u)}{u}\delta \cdot (\sigma\tau^{-1}\Delta'\tau^{-1}).
\]
\end{remark}

\subsection{Forced interior bracket}

For $\sigma$-derivations, the bracket defines a product on $\A \cdot \Delta$ since
\[
[a \cdot \Delta,b \cdot \Delta]_\sigma = (\sigma(a)\Delta(b)-\sigma(b)\Delta(a)) \cdot \Delta.
\]

For $a,b \in \A$, define $[a \cdot \Delta,b \cdot \Delta]'_{\tau,\sigma}$ by the same formula. The morphism $\tau$ does not appear in the formula, but the operator $\Delta$ is a $(\tau,\sigma)$-derivation. We get
\begin{equation}
[a \cdot \Delta,b \cdot \Delta]'_{\tau,\sigma} = (\sigma(a)\Delta(b)-\sigma(b)\Delta(a)) \cdot \Delta.
\end{equation}
In this case, $\A \cdot \Delta$ is closed under the bracket $[~,~]'_{\tau,\sigma}$ by definition. 

We want this bracket to satisfy a quasi-Jacobi identity. For $a,b,c \in \A$ we have,
\begin{align*}
\circlearrowleft_{a,b,c} & [\sigma(a) \cdot \Delta,[b \cdot \Delta,c \cdot \Delta]'_{\tau,\sigma}]'_{\tau,\sigma} \\
&= \circlearrowleft_{a,b,c} \bigg( \sigma^2(a) \Delta(\sigma(b)) \tau(\Delta(c)) - \sigma^2(a) \Delta(\sigma(c)) \tau(\Delta(b)) \\
& \hspace*{4em} - \sigma^2(b)\sigma(\Delta(c))\Delta(\sigma(a)) + \sigma^2(c)\sigma(\Delta(b))\Delta(\sigma(a)) \bigg) \cdot \Delta
\end{align*}
and computing the second bracket the other way,
\begin{align*}
\circlearrowleft_{a,b,c} & [\tau(a) \cdot \Delta,[b \cdot \Delta,c \cdot \Delta]'_{\tau,\sigma}]'_{\tau,\sigma} \\
&= \circlearrowleft_{a,b,c} \bigg( \sigma(\tau(a)) \Delta^2(c) \tau(\sigma(b)) - \sigma(\tau(a)) \Delta^2(b) \tau(\sigma(c)) \\
& \hspace*{4em} \sigma(\tau(a))\sigma(\Delta(c))\Delta(\sigma(b)) - \sigma(\tau(a))\sigma(\Delta(b))\Delta(\sigma(c)) \\
& \hspace*{4em} - \sigma(\Delta(c)) \sigma^2(b) \Delta(\tau(a)) + \sigma(\Delta(b)) \sigma^2(c) \Delta(\tau(a)) \bigg) \cdot \Delta.
\end{align*}

Suppose that there is an element $\delta \in \A$ such that for all $a \in \A,\ \Delta(\sigma(a)) = \delta \sigma(\Delta(a))$, then the second line in the first and second expression vanishes, doing the cyclic summation. If $\sigma$ and $\tau$ commutes \ie $\sigma \circ \tau = \tau \circ \sigma$, then the first line in the second expression vanishes, and finally, if we also have $\Delta(\tau(a)) = \delta \tau(\Delta(a))$ for all $a \in \A$, then the remaining first line in the first expression is the opposite of the last remaining line in the second expression when doing the cyclic summation.

So finally, with the conditions $\sigma \circ \tau = \tau \circ \sigma$, $\Delta(\sigma(a)) = \delta \sigma(\Delta(a))$ and $\Delta(\tau(a)) = \delta \tau(\Delta(a))$ for all $a \in \A$, we have
\[
\circlearrowleft_{a,b,c} [(\sigma+\tau)(a) \cdot \Delta,[b \cdot \Delta,c \cdot \Delta]'_{\tau,\sigma}]'_{\tau,\sigma} = 0
\]

We can sum this up in the following theorem.
\begin{theorem} \label{Thm:Forced-Bracket}
Let $\A$ be a commutative associative unital algebra and $\Delta$ be a $(\tau,\sigma)$-derivation of $\A$ with algebra morphisms $\sigma$, $\tau$ satisfying, for all $a \in \A$
\begin{equation} \label{Hom-Lie-commutation-condition}
\sigma(\tau(a)) = \tau(\sigma(a)) \qquad
\begin{aligned}
\Delta(\sigma(a)) &= \delta \sigma(\Delta(a)) \\
\Delta(\tau(a)) &= \delta \tau(\Delta(a))
\end{aligned}
\end{equation}
with $\delta \in \A$. The bracket $[~,~]'_{\tau,\sigma} : \A \cdot \Delta \times \A \cdot \Delta \to \A \cdot \Delta$ defined for $a,b \in \A$ by
\[
[a \cdot \Delta,b \cdot \Delta]'_{\tau,\sigma} = (\sigma(a)\Delta(b)-\sigma(b)\Delta(a)) \cdot \Delta
\]
endows the space $(\A \cdot \Delta,[~,~]'_{\tau,\sigma},\overline{\sigma + \tau})$ with a structure of Hom-Lie algebra.
\end{theorem}
Here we have extended maps $\alpha : \A \to \A$ on $\mathcal{D}_{\tau,\sigma}(\A)$ by $\overline{\alpha}(a \cdot D) = \alpha(a) \cdot D$.

\begin{cor}
Under the same hypothesis, by symmetry, the bracket $[~,~]''_{\tau,\sigma} : \A \cdot \Delta \times \A \cdot \Delta \to \A \cdot \Delta$ defined for $a,b \in \A$ by
\[
[a \cdot \Delta,b \cdot \Delta]''_{\tau,\sigma} = (\tau(a)\Delta(b)-\tau(b)\Delta(a)) \cdot \Delta
\]
endows the space $(\A \cdot \Delta,[~,~]''_{\tau,\sigma},\overline{\tau + \sigma})$ with a structure of Hom-Lie algebra.
\end{cor}

\bigskip
We look again the case when $\A$ is a unique factorization domain, with $\tau,\sigma : \A \to \A$ two different algebra endomorphisms. The \autoref{Thm:Rank-One-Module} again gives
\[
\mathcal{D}_{\tau,\sigma}(\A) = \A \cdot \Delta
\]
where $\Delta = \dfrac{\tau - \sigma}{g}$ and $g = \gcd((\tau - \sigma)(\A))$. 

\begin{prop}
If $\A$ is an unique factorization domain and $\tau,\sigma : \A \to \A$ are two different algebra endomorphisms with $\tau \sigma = \sigma \tau$, then equations \eqref{Hom-Lie-commutation-condition},
\begin{align*}
\Delta(\sigma(a)) &= \delta \sigma(\Delta(a)) \\
\Delta(\tau(a)) &= \delta \tau(\Delta(a))
\end{align*}
hold only if
\begin{equation}
\delta = \frac{\sigma(g)}{g} = \frac{\tau(g)}{g},
\end{equation}
where $g = \gcd((\tau-\sigma)(\A))$.
\end{prop}

\begin{proof}
Let $y \in \A$ and set
\[
x \coloneqq \frac{\tau-\sigma}{g}(y) = \frac{\tau(y)-\sigma(y)}{g}.
\]
Since $\tau$ and $\sigma$ commute, we have
\begin{align*}
\sigma(g)\sigma(x) &= \sigma(gx) = \sigma(\tau(y)) - \sigma^2(y) = (\tau - \sigma)(\sigma(y)) \\
\tau(g)\tau(x) &= \tau(gx) = \tau^2(y) - \tau(\sigma(y)) = (\tau - \sigma)(\tau(y))
\end{align*}
Since $g~|~\tau(g)-\sigma(g)$ by definition, if $g$ divides one of $\tau(g)$ or $\sigma(g)$, it will also divide the other. Suppose that it is the case, so dividing by $g$ in the previous equalities and substitute the expression for $x$ gives
\begin{align}
\frac{\sigma(g)}{g}\sigma\left( \frac{\tau-\sigma}{g}(y) \right) &= \frac{\tau-\sigma}{g}(\sigma(y)) \Leftrightarrow \frac{\sigma(g)}{g}\sigma(\Delta(y)) = \Delta(\sigma(y)) \\
\frac{\tau(g)}{g}\tau\left( \frac{\tau-\sigma}{g}(y) \right) &= \frac{\tau-\sigma}{g}(\sigma(y)) \Leftrightarrow \frac{\tau(g)}{g}\tau(\Delta(y)) = \Delta(\tau(y))
\end{align}
\end{proof}

So if $\sigma$ and $\tau$ commute and $\sigma(g)/g = \tau(g)/g = \delta$, we can use the \autoref{Thm:Forced-Bracket} to define a Hom-Lie algebra structure on $\mathcal{D}_{\tau,\sigma}(\A) = \A \cdot \Delta$ when $\A$ is a unique factorization domain.

\section{Examples with Laurent polynomials} \label{Sec:Examples}

\subsection{\texorpdfstring{$(p,q)$-deformations of Witt algebra}{(p,q)-deformations of Witt algebra}} \label{Sec:Witt-pq}

We consider the complex algebra $\A$ of Laurent polynomials in one variable $t$
\[ \A = \complex[t,t^{-1}]. \]

For $p,q \in \complex^*$ fixed, with $p \neq q$, define two endomorphisms of $\A$, $\tau$ and $\sigma$, by
\[ \tau(t)=pt, \quad \sigma(t) = qt. \]

From $\tau$ and $\sigma$ being endomorphisms it also follows that
\[ \tau(f(t))=f(pt), \qquad \sigma(f(t)) = f(qt) \quad \text{ for } f \in \A. \]
The algebra $\A$ is a unique factorization domain and by \autoref{Thm:Rank-One-Module} the set of $(\tau,\sigma)$-derivations $\mathcal{D}_{\tau,\sigma}(\A)$ is a free $\A$-module of rank one generated by the mapping
\[ D(f(t)) = \left(\frac{\tau-\sigma}{p-q}\right)(f(t)) = \frac{f(pt)-f(qt)}{p-q} = t \frac{f(pt)-f(qt)}{pt-qt} = t \frac{\tau(f(t))-\sigma(f(t))}{\tau(t)-\sigma(t)}. \]
Notice that the last equalities shows that $t^{-1}D = D_{p,q}$, the Jackson derivative. To see that $D$ is indeed a generator, we have that $g \coloneqq \gcd((\tau-\sigma)(\A)) = p - q$, because a $\gcd$ of $(\tau - \sigma)(\A)$ is any element of the form $c t^k$ with $c \in \complex^*$ and $k \in \intg$ since it divides $(\tau - \sigma)(t) = (p-q)t$ which is a unit (with $p \neq q$).

A general formula for $D$ acting on a monomial is as follows
\begin{equation}
D(t^n) = \frac{p^n - q^n}{p-q} t^n = [n]_{p,q} t^n
\end{equation}
where $[n]_{p,q}$ is the $(p,q)$-number defined for $n \in \intg$ and $p\neq q$ by $\displaystyle [n]_{p,q} \coloneqq \dfrac{p^n - q^n}{p-q}$. We will often omit the subscripts when the context makes it clear.

For $n \in \nat$ we have
\begin{equation}
[n] = \sum_{k=0}^{n-1} p^{n-1-k}q^k \qquad \text{and} \qquad [-n] = - (pq)^{-n} [n] = - \sum_{k=0}^{n-1} p^{-1-k}q^{k-n},
\end{equation}
which allows to define $[n]_{p,p} = n p^{n-1}$ for $n \in \intg$ when $q=p$.

Since $\sigma|_\complex = \tau|_\complex = id$, we have that $\delta = \sigma(\tau^{-1}(g))/g = 1$ so we can use the \autoref{Thm:Bracket-Gen-Der} to define the bracket, for $f,g \in \A$, by
\begin{align*}
[f \cdot D,g \cdot D] & \coloneqq (\sigma\tau^{-1}(f) \cdot D) \circ \tau^{-1} \circ (g \cdot D) - (\sigma\tau^{-1}(g) \cdot D) \circ \tau^{-1} \circ (f \cdot D) \\
& = (f(qp^{-1}t) \cdot D) \circ (g(p^{-1}t) \cdot \tau^{-1} \circ D) - (g(qp^{-1}t) \cdot D) \circ (f(p^{-1}t) \cdot \tau^{-1} \circ D).
\end{align*}

As a $\complex$-linear space, $\mathcal{D}_{\tau,\sigma}(\A)$ has a basis $\{d_n,\ n \in \intg\}$
\[
\mathcal{D}_{\tau,\sigma}(\A) = \bigoplus_{n \in \intg} \complex \cdot d_n \qquad \text{where} \qquad d_n = -t^n \cdot D.
\]
Then the bracket $[~,~]$ on $\mathcal{D}_{\tau,\sigma}(\A)$ can be rewritten as following, defined on generators by
\[
[d_n,d_m] = q^n p^{-n} d_n \circ \tau^{-1} \circ d_m - q^m p^{-m} d_m \circ \tau^{-1} \circ d_n
\]
with commutation relations
\[
[d_n,d_m] = \left( \frac{[n]}{p^n} - \frac{[m]}{p^m} \right) d_{n+m}.
\]
Indeed, we have
\begin{align*}
[d_n,d_m] & = \big(q^n p^{-n} (-t^n) D(-p^{-m}t^m) - q^m p^{-m} (-t^m) D(-p^{-n}t^n)\big) \cdot D \\
& = \left(q^m \frac{p^n-q^n}{p-q} - q^n \frac{p^m-q^m}{p-q}\right) \frac{1}{p^{n+m}} (-t^{n+m} \cdot D) \\
& = \frac{q^m p^{-m}-q^n p^{-n}}{p-q} d_{n+m} = \left(\frac{[n]}{p^n} - \frac{[m]}{p^m}\right) d_{n+m}.
\end{align*}
This bracket also satisfies skew-symmetry
\[
[d_m,d_n] = - [d_n,d_m]
\]
and the Hom-Jacobi identity
\[
\circlearrowleft_{n,m,k} \left[\left(q^n p^{-n}+1\right)d_n,\left[d_m,d_k\right]\right] = 0.
\]

We summarize our findings in a theorem.
\begin{theorem} \label{Thm:Witt-pq}
Let $\A = \complex[t,t^{-1}]$, $\tau(t)=pt,\ \sigma(t)=qt$ morphisms of $\A$ and $D = \dfrac{\tau-\sigma}{p-q}$ a $(\tau,\sigma)$-derivation of $\mathcal{D}_{\tau,\sigma}(\A)$. Then the $\complex$-linear space
\[
W_{p,q} = \bigoplus_{n \in \intg} \complex \cdot d_n,
\]
where $d_n = -t^n \cdot D$, can be endowed of a structure of Hom-Lie algebra with the bracket defined on generators by
\begin{equation}
[d_n,d_m] = q^n p^{-n} d_n \circ \tau^{-1} \circ d_m - q^m p^{-m} d_m \circ \tau^{-1} \circ d_n
\end{equation}
with commutation relations
\begin{equation}
[d_n,d_m] = \left( \frac{[n]}{p^n} - \frac{[m]}{p^m} \right) d_{n+m},
\end{equation}
and with the twist morphism $\alpha = id + \overline{\sigma} \overline{\tau^{-1}} : W_{p,q} \to W_{p,q}$ defined by
\begin{equation}
\alpha(d_n) = (1+(q/p)^n)d_n.
\end{equation}
The bracket satisfies skew-symmetry $[d_m,d_n] = - [d_n,d_m]$ and the Hom-Jacobi identity
\begin{equation}
\circlearrowleft_{n,m,k} \left[\left(1+(q/p)^n\right)d_n,\left[d_m,d_k\right]\right] = 0.
\end{equation}
\end{theorem}

It gives an example of a $(p,q)$-deformation of Witt algebra. When $p=1$ we recover the $q$-deformed Witt algebra defined in \cite{HLS06} and for $p=q=1$, we obtain the usual Witt algebra.

\bigskip

\begin{remark}
The presentation of $(p,q)$-deformations of the Witt algebra by generators $\{d_n\}_{n \in \intg}$ and relations $[d_n,d_m] = \left( \frac{[n]_{p,q}}{p^n} - \frac{[m]_{p,q}}{p^m} \right) d_{n+m}$ with $(p,q)$-numbers $[n]_{p,q} = \frac{p^n-q^n}{p-q}$ is similar to the presentation of Witt $q$-deformations with generators $\{d_n\}_{n \in \intg}$ and relations $[d_n,d_m] = (\{n\}_q - \{m\}_q) d_{n+m}$ with $q$-numbers $\{n\}_q = \frac{1-q^n}{1-q}$.

Algebraically, the identities obtained come from the fact that
\begin{equation} \label{pq-q-numbers}
[n]_{p,q} = \frac{p^n-q^n}{p-q} = \frac{1-(q/p)^n}{1-q/p} \frac{p^n}{p} = \frac{p^n}{p} \{n\}_{q/p} \qquad \Rightarrow \qquad \frac{[n]_{p,q}}{p^n} = \frac{1}{p} \{n\}_{q/p},
\end{equation}
so replacing $r$-numbers $\{n\}_r$ by $(p,q)$-numbers $[n]_{p,q}$ with $r = q/p$ formally gives similar relations.

However, $q$-deformations of the Witt algebra were defined using $\sigma$-derivations, but the $(p,q)$-deformations arise by the \autoref{Thm:Bracket-Gen-Der} which gives a quasi-Lie structure on $(\tau,\sigma)$-derivations.
\end{remark}

Since the endomorphisms $\tau$ and $\sigma$ in this example are just the multiplications by $p$ or $q$, the $(p,q)$-deformations of Witt algebra are isomorphic to $q/p$-deformations of Witt algebra (as Hom-Lie algebras).
\begin{prop} \label{Prop:Wpq-iso-Wr}
The Hom-Lie algebra
\[
\left(W_{p,q},\quad [d_n,d_m]_{p,q} = \left( \frac{[n]_{p,q}}{p^n} - \frac{[m]_{p,q}}{p^m} \right) d_{n+m},\quad \alpha_{p,q}(d_n) = (1+(q/p)^n) d_n \right)
\]
is isomorphic to
\[
\left(W_{q/p},\quad [d_n,d_m]_{q/p} = \left( \{n\}_{q/p} - \{m\}_{q/p} \right) d_{n+m},\quad \alpha_{q/p}(d_n) = (1+(q/p)^n) d_n \right).
\]
\end{prop}
The generators $d_n = -t^n \cdot D$ in the two cases are not the same: for $W_{q/p}$ the map $D$ is a $\rho$-derivation, for $W_{p,q}$ the map $D$ is a $(\tau,\sigma)$-derivation. Since the definition of the generators $d_n$ does not occur hereinafter, we abusively keep the same notation.

\begin{proof}
Let $\varphi : W_{q/p} \to W_{p,q}$ be a linear map defined for $n \in \intg$ by $\varphi(d_n) = c_n d_n$, with $c_n \neq 0$. We have that $\varphi(\alpha_{q/p}(d_n)) = \alpha_{p,q}(\varphi(d_n))$ and the equation
\begin{gather*}
[\varphi(d_n),\varphi(d_m)]_{p,q} = \varphi([d_n,d_m]_{q/p})
\intertext{is equivalent, using identities \eqref{pq-q-numbers}, to}
c_n c_m \left( \frac{[n]_{p,q}}{p^n} - \frac{[m]_{p,q}}{p^m} \right) = \left(\{n\}_{q/p} - \{m\}_{q/p}\right) c_{n+m} \quad \Leftrightarrow \quad p c_{n+m} = c_n c_m.
\end{gather*}
Taking $n=m=0$, we get $p c_0 = c_0^2 \Rightarrow c_0 = p$. By induction, we obtain for $n \in \nat$
\[
c_n = \frac{c_1^n}{p^{n-1}} \quad \text{and}\quad c_{-n} = \frac{c_{-1}^{-n}}{p^{-n-1}}.
\]
This gives a family of isomorphisms between $W_{q/p}$ and $W_{p,q}$. For $c_1 = c_{-1} = p$, we have $\varphi(d_n) = p d_n$ is just the multiplication by $p$.
\end{proof}

\bigskip

An other bracket-product on $\mathcal{D}_{\tau,\sigma}(\A)$ can be defined by
\[ [f\cdot D, g \cdot D]' = (\sigma(f)D(g)-\sigma(g)D(f)) \cdot D \]
Since $\sigma \circ \tau = \tau \circ \sigma$, and $\sigma(p-q) = \tau(p-q) = p-q$ implies $\delta = 1$, we have $\sigma \circ D = D \circ \sigma$ and $\tau \circ D = D \circ \tau$ so all conditions for equation \eqref{Hom-Lie-commutation-condition} are fulfilled. Since the bracket product is also skew-symmetric it follows from the \autoref{Thm:Forced-Bracket} that $\mathcal{D}_{\tau,\sigma}(\A)$ is a Hom-Lie algebra, with bracket given on generators $d_n = -t^n \cdot D$ by
\begin{align*}
[d_n, d_m]' &= (q^m [n] - q^n [m]) d_{n+m} = (p^m [n] - p^n [m]) d_{n+m}.
\end{align*}
Extend maps $\alpha : \A \to \A$ on $\mathcal{D}_{\tau,\sigma}(\A)$ by $\overline{\alpha}(a \cdot D) = \alpha(a) \cdot D$. We can then state the following result.
\begin{theorem}
Let $\A = \complex[t,t^{-1}]$, $\tau(t)=pt,\ \sigma(t)=qt$ morphisms of $\A$ and $D = \dfrac{\tau-\sigma}{p-q}$ a $(\tau,\sigma)$-derivation of $\mathcal{D}_{\tau,\sigma}(\A)$. Then the $\complex$-linear space
\[
W_{p,q} = \bigoplus_{n \in \intg} \complex \cdot d_n,
\]
where $d_n = -t^n \cdot D$, can be endowed of a structure of Hom-Lie algebra with the bracket defined on generators by
\begin{align}
[d_n, d_m]' &= (q^m [n] - q^n [m]) d_{n+m} = (p^m [n] - p^n [m]) d_{n+m}.
\end{align}
and with the twist morphism $\overline{\sigma + \tau} : W_{p,q} \to W_{p,q}$ defined by
\begin{equation}
\overline{\sigma + \tau}(d_n) = (p^n+q^n)d_n.
\end{equation}
The bracket satisfies skew-symmetry $[d_m,d_n] = - [d_n,d_m]$ and the Hom-Jacobi identity
\begin{equation}
\circlearrowleft_{n,m,k} [(p^n+q^n)d_n,[d_m,d_k]']' = 0.
\end{equation}
\end{theorem}

\begin{remark}
This example is closely related to the previous one. Indeed, defining a morphism on $W_{p,q}$ by
\begin{align*}
\varphi : \bigoplus_{n \in \intg} \complex \cdot d_n &\to \bigoplus_{n \in \intg} \complex \cdot d_n \\
d_n & \mapsto p^n d_n
\end{align*}
we have that $(W_{p,q},[~,~]' = \varphi \circ [~,~],\overline{\sigma+\tau} = \varphi \circ \alpha)$ is the Hom-Lie algebra obtained by twisting (as in \autoref{Thm:TwistAssLie}) the previous $(p,q)$-Witt Hom-Lie algebra by the morphism $\varphi$.

Although these Hom-Lie algebras are twist-equivalent, they are not isomorphic just by a scale changing relation. Define
\begin{align*}
\varphi : (W_{p,q},[~,~],\alpha) &\to (W_{p,q},[~,~]',\alpha') \\
\varphi(d_n) &= c_n d_{\nu(n)}
\end{align*}
with $\nu : \intg \to \intg$ a permutation of the integers, and $c_n \neq 0$ for $n \in \intg$. Suppose that $\varphi$ is an isomorphism, then
\begin{gather}
[\varphi(d_n),\varphi(d_m)]' = \varphi([d_n,d_m]) \notag \\
\Leftrightarrow c_n c_m \left( p^{\nu(m)} [\nu(n)] - p^{\nu(n)} [\nu(m)] \right) d_{\nu(n)+\nu(m)} = \frac{1}{p^{n+m}} (p^m [n] - p^n [m]) c_{n+m} d_{\nu(n+m)}, \notag
\intertext{since these expressions are multiple of basis elements, we must have}
d_{\nu(n)+\nu(m)} = d_{\nu(n+m)} \Rightarrow \nu(n+m) = \nu(n)+\nu(m) \Rightarrow \nu(n) = n \nu(1) \notag
\intertext{and}
c_n c_m p^{n+m} (p^{\nu(n)}q^{\nu(m)} - p^{\nu(m)}q^{\nu(n)}) = c_{n+m} (p^n q^m - p^m q^n). \label{Witt-not-isomorphic}
\end{gather}
The isomorphism $\varphi$ also intertwine the twist maps which gives
\begin{gather}
\alpha'(\varphi(d_n)) = \varphi(\alpha(d_n)) \notag \\
\Leftrightarrow \cancel{c_n} (p^\nu(n) + q^{\nu(n)}) d_{\nu(n)} = (1+(q/p)^n) \cancel{c_n} d_{\nu(n)} \notag \\
\Leftrightarrow p^n (p^{\nu(n)} + q^{\nu(n)}) = p^n + q^n \label{Witt-not-isomorphic-twists}
\end{gather}
If $\nu \equiv id$ ($\nu(1)=1$), then equation \ref{Witt-not-isomorphic} gives
\[
c_n c_m p^{n+m} = c_{n+m}.
\]
Since for $n \in \nat,\ c_n \neq 0$, for $n=m=0$, we get $c_0 = 1$, and for $n=1,m=0$, we obtain $c_1 c_0 p = c_1$ which is impossible unless $p=1$. In this case, we recover for both algebras $W_{1,q} = W_q$, a $q$-deformation of Witt algebra. For $p \neq 1$, equations \eqref{Witt-not-isomorphic} and \eqref{Witt-not-isomorphic-twists} show that the two Hom-Lie algebras are not isomorphic by a mere scale change relation (with $\nu \equiv id$). An isomorphism which rescale and permutes the generators $d_n$ of $W_{p,q}$ must satisfy equations \eqref{Witt-not-isomorphic} and \eqref{Witt-not-isomorphic-twists} for all $n \in \intg$.
\end{remark}

\subsection{\texorpdfstring{$(p,q)$-deformations of $\mathfrak{sl}(2)$}{(p,q)-deformations of sl(2)}}

We consider now the same algebra $\A = \complex[t,t^{-1}]$ and morphisms $\tau,\sigma : \A \to \A$ uniquely defined by $\tau(t) = pt$ and $\sigma(t) = qt$ with $p,q \in \complex^*$, $p\neq q$; but this time, we choose $\partial = \dfrac{\tau - \sigma}{t(p-q)}$ as the generator of the $\A$-module $\mathcal{D}_{\tau,\sigma}(\A) = \A \cdot \partial$, with $\partial(t^n) = [n]t^{n-1}$, and we restrict the bracket
\begin{align*}
[a \cdot \partial,b \cdot \partial] & \coloneqq (\sigma\tau^{-1}(a) \cdot \partial) \circ \tau^{-1} \circ (b \cdot \partial) - (\sigma\tau^{-1}(b) \cdot \partial) \circ \tau^{-1} \circ (a \cdot \partial) \\
& = \big(\sigma\tau^{-1}(a) \partial\tau^{-1}(b) - \sigma\tau^{-1}(b) \partial\tau^{-1}(a)\big) \cdot \partial
\end{align*}
given by \autoref{Thm:Bracket-Gen-Der} to the space $S$ generated by
\[
e \coloneqq \partial \qquad f \coloneqq - t^2 \cdot \partial \qquad h \coloneqq - 2t \cdot \partial
\]
with $S = \operatorname{LinSpan}(e,f,h)$.

We have the following brackets on these elements
\begin{subequations} \label{brackets-efh}
\begin{align}
[h,e] &= [-2t \cdot \partial, \partial] = 2 \partial(p^{-1}t) \cdot \partial = 2p^{-1} \partial = 2p^{-1} e \\
[h,f] &= [-2t \cdot \partial, -t^2 \cdot \partial] = \big( 2qp^{-1}t \partial(p^{-2}t^2) - q^2p^{-2}t^2 \partial(2p^{-1}t) \big) \cdot \partial = 2qp^{-2}t^2 \cdot \partial = -2qp^{-2} f \\
[e,f] &= [\partial,-t^2 \cdot \partial] = -\partial(p^{-2}t^2) \cdot \partial = -p^{-2}(p+q)t \cdot \partial = \frac{p+q}{2p^2}h.
\end{align}
\end{subequations}

Since the gcd has changed in the choosen generator, the element $\delta$ also changes according to the \autoref{Rmk:BaseChangeRelation}. We have in this case $\delta = qp^{-1}$ and the Hom-Jacobi relation coming from the $(p,q)$-Witt Hom-Lie algebra writes
\begin{gather*}
\circlearrowleft_{e,f,h} [\overline{\sigma} \overline{\tau^{-1}}(e)+qp^{-1} e,[f,h]] \\
= [e,2qp^{-2}f] + [q^2 p^{-2}f,2p^{-1}e] + \underbrace{\left[q p^{-1}h,\frac{p+q}{2p^2}h\right]}_{=0} \\
+ qp^{-1} [e,2qp^{-2}f] + qp^{-1}[f,2p^{-1}e] + qp^{-1}\underbrace{\left[h,\frac{p+q}{2p^2}h\right]}_{=0} = 0.
\end{gather*}

\begin{remark} \label{Rmk:gen-efh-dn}
We can also use the generators $d_n = - t^n \cdot \partial$ of $W_{p,q}$. The elements $e,f,h$ can be written
\[
e = \partial = - d_0 \qquad f = -t^2 \cdot \partial = d_2 \qquad h = -2t \cdot \partial = 2 d_1.
\]
Since $\partial = \frac{1}{t} D$, the bracket in this case is $[d_n,d_m] = ([n]-[m]) d_{n+m-1}$, we can then recover the brackets of equations \eqref{brackets-efh}.
\end{remark}

We obtain then the following result.
\begin{theorem} \label{sl2pq}
Let $\A = \complex[t,t^{-1}]$, $\tau(t)=pt,\ \sigma(t)=qt$ morphisms of $\A$ and $\partial = \dfrac{\tau - \sigma}{t(p-q)}$ a $(\tau,\sigma)$-derivation of $\mathcal{D}_{\tau,\sigma}(\A)$. Then the space $S \subset \mathcal{D}_{\tau,\sigma}(\A)$ generated by $e \coloneqq \partial$, $f \coloneqq - t^2 \cdot \partial$, $h \coloneqq - 2t \cdot \partial$ is a Hom-Lie algebra, with brackets and twist morphism given by
\begin{equation*}
\begin{aligned}[l]
[h,e]_{p,q} &= 2p^{-1} e, \\
[h,f]_{p,q} &= -2qp^{-2} f, \\
[e,f]_{p,q} &= \dfrac{p+q}{2p^2} h,
\end{aligned}
\hspace{3cm}
\begin{aligned}[r]
\alpha_{p,q}(e) &= (1+qp^{-1})e, \\
\alpha_{p,q}(f) &= qp^{-1}(1+qp^{-1}) f, \\
\alpha_{p,q}(h) &= 2qp^{-1} h.
\end{aligned}
\end{equation*}
This Hom-Lie algebra will be noted $\mathfrak{sl}(2)_{p,q}$.
\end{theorem}

For $p=1$ ($\tau = id$), we recover the Jackson $\mathfrak{sl}(2)_q$ Hom-Lie algebra mentioned in \cite{LS07}, and for $p=q=1$, we obtain the usual $\mathfrak{sl}(2)$ Lie algebra.

Taking $p=1$ in the preceding theorem, the Jackson $\mathfrak{sl}(2)_q$ Hom-Lie algebra can be presented as the space $\tilde{S} \subset \mathcal{D}_\sigma(\A)$ generated by $e \coloneqq \partial$, $f \coloneqq - t^2 \cdot \partial$, $h \coloneqq - 2t \cdot \partial$, where this time $\partial = \dfrac{id - \sigma}{t(1-q)}$ is a $\sigma$-derivation, with brackets and twist morphism given by
\begin{equation*}
\begin{aligned}[l]
[h,e]_q &= 2 e, \\
[h,f]_q &= -2q f, \\
[e,f]_q &= \dfrac{1+q}{2} h,
\end{aligned}
\hspace{3cm}
\begin{aligned}[r]
\alpha_q(e) &= (1+q)e, \\
\alpha_q(f) &= q(1+q) f, \\
\alpha_q(h) &= 2q h.
\end{aligned}
\end{equation*}

\begin{prop}
The Hom-Lie algebra $\mathfrak{sl}(2)_{p,q}$ is isomorphic to $\mathfrak{sl}(2)_{q/p}$.
\end{prop}

\begin{proof}
As in \autoref{Prop:Wpq-iso-Wr} and keeping in mind the expressions of \autoref{Rmk:gen-efh-dn} for $e,f,h$, we search an isomorphism $\varphi : \mathfrak{sl}(2)_r \to \mathfrak{sl}(2)_{p,q}$ of the form $\varphi(d_n) = c_n d_n$, with $c_n \neq 0$, for $n \in \{0,1,2\}$. We then obtain
\begin{align*}
\varphi([h,e]_r) &= [\varphi(h),\varphi(e)]_{p,q} & \Leftrightarrow && \quad -2 c_0 d_0 &= -2 c_0 c_1 p^{-1} d_0 \quad && \Leftrightarrow & c_1 &= p \\
\varphi([h,f]_r) &= [\varphi(h),\varphi(f)]_{p,q} & \Leftrightarrow && \quad -2 r c_2 d_2 &= -2 c_1 c_2 q p^{-2} d_2 \quad && \Leftrightarrow & r &= q/p \\
\varphi([e,f]_r) &= [\varphi(e),\varphi(f)]_{p,q} & \Leftrightarrow  && \qquad 2 c_1 d_1 &= c_0 c_2 \frac{p+q}{p^2} d_1 \quad && \Leftrightarrow & c_0 c_2 &= \frac{2p^3}{p+q},
\end{align*}
and $\varphi(\alpha_r(x)) = \alpha_{p,q}(\varphi(x))$ for $x = e,f,h$ only gives $r = q/p$. Taking $c_0 = p$ and $c_2 = \frac{2p^2}{p+q}$ for example, we get an isomorphism $\varphi : \mathfrak{sl}(2)_{q/p} \to \mathfrak{sl}(2)_{p,q}$ given by
\[
\varphi(e) = p e, \quad \qquad \varphi(f) = \frac{2p^2}{p+q} f, \quad \qquad \varphi(h) = p h.
\]

\end{proof}

\subsection{\texorpdfstring{Case of $(\sigma,\sigma)$-derivations}{Case of (sigma,sigma)-derivations}}

We keep the algebra $\A = \complex[t,t^{-1}]$ and morphism $\sigma : \A \to \A$ uniquely defined by $\sigma(t) = pt$ with $p \in \complex^*$. If we take $\tau = \sigma$, the $(\sigma,\sigma)$-derivation $\partial$ on $\mathcal{D}_{\sigma,\sigma}(\A)$ is well defined, given by $\partial(t^n) = [n]t^{n-1} = n (pt)^{n-1}$. The module $\mathcal{D}_{\sigma,\sigma}(\A)$ is still generated by $\partial$, because for any $D \in \mathcal{D}_{\sigma,\sigma}(\A)$ and $n \in \intg$, we have $D(t^n) = [n] t^{n-1} D(t)$ so $D = D(t) \cdot \partial$.

In this case, for the bracket defined for $f,g \in \A$ by
\begin{gather*}
[f \cdot \partial,g \cdot \partial] \coloneqq (f \cdot \partial) \circ (\sigma^{-1}(g) \cdot \sigma^{-1} \circ \partial) - (g \cdot \partial) \circ (\sigma^{-1}(f) \cdot \sigma^{-1} \circ \partial)
\intertext{and for generators of $\mathcal{D}_{\sigma,\sigma}(\A)$ as a $\complex$ linear space}
d_n \coloneqq - t^n \cdot \partial,
\end{gather*}
the bracket on generators writes
\[
[d_n,d_m] = d_n \circ \sigma^{-1} \circ d_m - d_m \circ \sigma^{-1} \circ d_n = \frac{n-m}{p} d_{n+m-1}
\]
and the \autoref{Thm:Bracket-Gen-Der} gives that $(W_{p,p},[~,~])$ is a Lie algebra, isomorphic to the classical Witt algebra $W = W_{1,1}$.
\begin{prop}
The linear map
\begin{align*}
\varphi : W &\to W_{p,p} \\
d_n & \mapsto p^{n+1} d_{n+1}.
\end{align*}
is a Lie algebra isomorphism.
\end{prop}

\begin{remark}
Note that if we choose $D = t \partial$ as generator of $\mathcal{D}_{\sigma,\sigma}(\A)$, the bracket writes
\[
[d_n,d_m] = \frac{n-m}{p} d_{n+m}
\]
and the isomorphism between $W$ and $W_{p,p}$ is the multiplication by $p$.

In this case, the Hom-Lie algebra given by \autoref{Thm:Forced-Bracket} has bracket relation
\[
[d_n,d_m] = (n-m)p^{n+m-1} d_{n+m}
\]
and is obtained by twisting the Lie algebra $(W_{p,p}, [d_n,d_m]=\frac{n-m}{p}d_{n+m},id)$ by the morphism
\begin{align*}
\varphi : W_{p,p} &\to W_{p,p} \\
d_n & \mapsto p^n d_n.
\end{align*}
\end{remark}

\bigskip

The situation for $\mathfrak{sl}(2)_{(p,p)}$ algebras is the same. Taking $q=p$ in the \autoref{sl2pq}, we obtain that $\mathfrak{sl}(2)_{(p,p)}$ is a Lie algebra with brackets
\[
[h,e] = 2 p^{-1}e, \quad \qquad [h,f] = - 2 p^{-1} f, \quad \qquad [e,f] = p^{-1}h
\]
isomorphic to the classical $\mathfrak{sl}(2)$ algebra.
\begin{prop}
The Lie algebra $\mathfrak{sl}(2)_{(p,p)}$ is isomorphic to $\mathfrak{sl}(2)$, the isomorphism is given by multiplication by $p$.
\end{prop}

We have $\partial(\tau(t^n)) = p \tau(\partial(t^n))$ and $\partial(\sigma(t^n)) = q \sigma(\partial(t^n))$. In this case with $p=q$, conditions \eqref{Hom-Lie-commutation-condition} are satisfied with $\delta = p = q$, and the \autoref{Thm:Forced-Bracket} gives a Hom-Lie algebra with brackets
\[
[h,e]' = 2 e, \quad \qquad [h,f]' = -2p^2f, \quad \qquad [e,f]' = ph
\]
and Hom-Jacobi relation (since $\overline{\tau} = \overline{\sigma}$, $\overline{\tau+\sigma} = 2 \overline{\sigma}$)
\begin{gather*}
[2\overline{\sigma}(e),[f,h]']' + [2\overline{\sigma}(f),[h,e]']' + [2\overline{\sigma}(h),[e,f]']' \\
= 2 \big( [e,2p^2f]' + [p^2 f,2e]' + \underbrace{[ph,ph]'}_{=0} \big) = 0.
\end{gather*}
This Hom-Lie algebra is the twist of the classical $\mathfrak{sl}(2)$ algebra by the morphism $\varphi$ given by
\[
\varphi(e) = e, \quad \qquad \varphi(f) = p^2 f, \quad \qquad \varphi(h) = ph.
\]

\bigskip

\begin{landscape}
\subsection{Deformations summarized in diagrams} \label{Sec:summarized-diagrams}

We can summarize the situation of the different Hom-Lie algebras obtained by $(p,q)$-deformations of the Witt algebra with the following diagram.

\begin{equation*}
\xymatrixcolsep{5pc}
\xymatrix{
\left( W_{q/p},\; [d_n,d_m] = \left( \{n\}_{q/p} - \{m\}_{q/p} \right) d_{n+m},\; \alpha(d_n) = (1+(q/p)^n)d_n \right) \ar@{.>}[r]^-{q=p}
\ar[dd]_{\text{Hom-Lie \enspace} \\ \rotatebox{90}{$\cong$}}^{d_n \mapsto p d_n} & (W,\quad [d_n,d_m] = (n-m) d_{n+m},\quad 2 id)
\ar[dd]_{\text{Lie \enspace} \\ \rotatebox{90}{$\cong$}}^{d_n \mapsto p d_n} \\
& \\
\left( W_{p,q},\quad [d_n,d_m] = \left( \frac{[n]}{p^n} - \frac{[m]}{p^m} \right) d_{n+m},\quad \alpha(d_n) = (1+(q/p)^n)d_n \right) \ar@{.>}[r]^-{q=p}
\ar@{~>}[dd]_{\text{twist-equivalent}}^{\varphi(d_n)=p^n d_n} &
\left( W_{p,p},\quad [d_n,d_m] = \frac{n-m}{p} d_{n+m},\quad 2 id \right)
\ar@{~>}[dd]_{\text{twist-equivalent}}^{\varphi(d_n)=p^n d_n} \\
& \\
\left( W_{p,q},\quad [d_n,d_m]' = \left( p^m [n] - p^n [m] \right) d_{n+m},\;\; \alpha'(d_n) = (p^n+q^n)d_n \right) \ar@{.>}[r]^-{q=p} &
\left( W_{p,p},\;\; [d_n,d_m]' = (n-m) p^{n+m-1} d_{n+m},\;\; \alpha'(d_n) = 2 p^n d_n \right)
}
\end{equation*}

\bigskip

Here the isomorphism on the upper right is a Lie algebra isomorphism between the Witt algebra $(W,\enspace [d_n,d_m] = (n-m) d_{n+m})$ and $(W_{p,p},\enspace [d_n,d_m] = \frac{n-m}{p} d_{n+m})$. \\

The situation for deformations of $\mathfrak{sl}(2)$ is slightly different, since the \autoref{Thm:Forced-Bracket} can't be applied unless $p=q$.
\begin{equation*}
\xymatrixcolsep{5pc}
\xymatrix{
\left( \mathfrak{sl}(2)_{q/p},\; [~,~]_{q/p},\; \alpha_{q/p} \right) \ar@{.>}[r]^-{q=p}
\ar[dd]_{\text{Hom-Lie \enspace} \\ \rotatebox{90}{$\cong$}}^{\txt{\small $e \mapsto p e$ \\ \small $f \mapsto \frac{2p^2}{p+q} f$ \\ \small $h \mapsto ph$}} & (\mathfrak{sl}(2),\quad [~,~])
\ar[dd]_{\text{Lie \enspace} \\ \rotatebox{90}{$\cong$}}^{\txt{\small $e \mapsto p e$ \\ \small $f \mapsto p f$ \\ \small $h \mapsto ph$}} \ar@{~>}[r]^-{\text{twist-equivalent}}
&
\left( \mathfrak{sl}(2)_{p,p},\quad [~,~]_{p,p}',\quad \tilde{\alpha} \right) \\
& \\
\left( \mathfrak{sl}(2)_{p,q},\quad [~,~]_{p,q},\quad \alpha_{p,q} \right) \ar@{.>}[r]^-{q=p} &
\left( \mathfrak{sl}(2)_{p,p},\quad [~,~]_{p,p} \right)
}
\end{equation*}
\end{landscape}

\subsection{\texorpdfstring{$(p,q)$-deformations of the Virasoro algebra}{(p,q)-deformations of the Virasoro algebra}}

Let $\A = \complex[t,t^{-1}]$, $\tau$ and $\sigma$ be the algebra endomorphisms on $\A$ satisfying $\tau(t) = pt,\ \sigma(t) = qt$, where $p,q \in \complex^*$, with $p \neq q$ and $q/p$ not a root of unity, and set $L = \mathcal{D}_{\tau,\sigma}(\A) = W_{p,q}$. Then $L$ can be given the structure of a Hom-Lie algebra $(L,\alpha)$ as described in \autoref{Sec:Witt-pq} \autoref{Thm:Witt-pq}.

As in the case of $q$-deformations of the Witt algebra extended into $q$-deformations of the Virasoro algebra exposed in \cite{HLS06}, we obtain the following results.

\begin{theorem}
Every non-trivial one-dimensional central extension of the Hom-Lie algebra $(\mathcal{D}_{\tau,\sigma}(\A),\alpha)$, where $\A = \complex[t,t^{-1}]$, is isomorphic to the Hom-Lie algebra $\operatorname{Vir}_{p,q} = (\hat{L}, \hat{\alpha})$, where $\hat{L}$ has basis $\{L_n, n \in \intg\} \cup \{\boldsymbol{c}\}$ and bracket relations
\begin{gather}
[\boldsymbol{c},\hat{L}]_{\hat{L}} \coloneqq 0, \\
[L_n,L_m]_{\hat{L}} \coloneqq \left( \frac{[n]}{p^n} - \frac{[m]}{p^m} \right) L_{n+m} + \delta_{n+m,0} \frac{(q/p)^{-n}}{6(1+(q/p)^n)} \frac{[n-1]}{p^{n-1}} \frac{[n]}{p^n} \frac{[n+1]}{p^{n+1}} \boldsymbol{c},
\end{gather}
and $\hat{\alpha} : \hat{L} \to \hat{L}$ is the endomorphism of $\hat{L}$ defined by
\begin{equation}
\hat{\alpha}(L_n) = (1+(q/p)^n) L_n, \qquad \hat{\alpha}(\boldsymbol{c}) = \boldsymbol{c}.
\end{equation}
Moreover, there exists a non-trivial central extension of $(\mathcal{D}_{\tau,\sigma}(\A),\alpha)$ by $(\complex,id_\complex)$.
\end{theorem}

\begin{proof}[Sketch of the proof]
The proof is done following the same lines that the proof of \cite[Theorems 39 and 40 p.356]{HLS06}. It uses uniqueness (\autoref{Thm:ExtHomLie-uniqueness}) and existence (\autoref{Thm:ExtHomLie-existence}) results for Hom-Lie algebras. The idea is to start from an arbitrary section $s : L \to \hat{L}$ and modify it so that the modified corresponding `$2$-cocycle' $g_{\tilde{s}} \in \bigwedge^2(L,\complex)$ has an expression as simple as possible. The generators $L_n$ of $\operatorname{Vir}_{p,q}$ are then defined to be the image of the generators $d_n$ of $L = W_{p,q}$ by the modified section $\tilde{s}$, and they satisfy the given equation due to properties of $g_{\tilde{s}}$. The formulas obtained are similar than for $q$-deformations because of identities \eqref{pq-q-numbers}.
\end{proof}

\subsection{\texorpdfstring{Another example of $(\tau,\sigma)$-derivation on Laurent polynomials}{Another example of (tau,sigma)-derivation on Laurent polynomials}} \label{Sec:another-ex}

We consider again the complex algebra $\A$ of Laurent polynomials in one variable $t$ \( \A = \complex[t,t^{-1}]\). This time, we consider $\tau$ and $\sigma$ endomorphisms of $\A$, defined by
\[ \sigma(t) = qt \quad \text{($q \in \complex$)}, \qquad \tau(t)=t^{-1}, \]
or, written on an element $f \in \A$
\[ \sigma(f(t)) = f(qt) \qquad \tau(f(t)) = f(t^{-1}). \]
The algebra $\A$ is a unique factorization domain and by \autoref{Thm:Rank-One-Module} the set of $(\tau,\sigma)$-derivations $\mathcal{D}_{\tau,\sigma}(\A)$ is a free $\A$-module of rank one generated by the mapping
\[
D = \frac{\tau - \sigma}{g},
\]
where $g = \gcd((\tau - \sigma)(\A))$. Since
\[
\frac{t^{-n} - q^n t^n}{t^{-1}-qt} = t^{-n+1}(1 + qt^2 + \dotsb + (qt^2)^{n-1}),
\]
we have that $(\tau - \sigma)(t) = t^{-1} - qt~|~t^{-n} - q^n t^n = (\tau - \sigma)(t^n)$. Hence $t^{-1} - qt~|~(\tau - \sigma)(f(t))$ for all $f \in \A$, and since $\gcd((\tau - \sigma)(\A))~|~(\tau - \sigma)(t)$, $t^{-1} - qt$ is associated to $\gcd((\tau - \sigma)(\A)$, so we can choose it as $\gcd$: $g \coloneqq t^{-1} - qt = \gcd((\tau - \sigma)(\A))$. Note that here, the $\gcd$ $g = t^{-1} - qt$ is not an unit of $\A$.

The $(\tau,\sigma)$-derivation $D$ acts then on a monomial $f(t) = t^n$ by the following formula:
\[ D(f(t)) = \left(\frac{\tau-\sigma}{g}\right)(f(t)) = \frac{f(t^{-1})-f(qt)}{g} = \frac{t^{-n} - q^n t^n}{t^{-1}-qt}. \]
The element $\delta$ of \autoref{Prop:Bracket-Gen-Der-UFD} is computed as
\[
\delta = \frac{\sigma\tau^{-1}(g)}{g} = \frac{\sigma(t-qt^{-1})}{t^{-1}-qt} = \frac{qt - t^{-1}}{t^{-1} - qt} = -1.
\]
The bracket of \autoref{Thm:Bracket-Gen-Der}
\[
[a \cdot \Delta,b \cdot \Delta]_{\tau,\sigma} = \Big( (\sigma \circ \tau^{-1})(a) (\Delta \circ \tau^{-1})(b) - (\sigma \circ \tau^{-1})(b) (\Delta \circ \tau^{-1})(a) \Big) \cdot \Delta
\]
gives on monomials
\begin{gather*}
[- t^n \cdot \Delta, - t^m \cdot \Delta]_{\tau,\sigma} = \Big( q^{-n} t^{-n}  \frac{t^m - q^{-m} t^{-m}}{t^{-1} - qt} - q^{-m} t^{-m} \frac{t^n - q^{-n}t^{-n}}{t^{-1} - qt} \Big) \cdot \Delta \\
= \frac{1}{t^{-1} - qt} \big( q^{-n} t^{m-n} - \cancel{q^{-n-m} t^{-n-m}} - q^{-m} t^{n-m} + \cancel{q^{-m-n}t^{-m-n}} \big) \\
= \frac{q^{-m}}{t^{-1} - qt} (- t^{n-m} \cdot \Delta) - \frac{q^{-n}}{t^{-1} - qt} (- t^{m-n} \cdot \Delta).
\end{gather*}
Setting $d_n \coloneqq - t^n \cdot \Delta$ as before as generators of $\mathcal{D}_{\tau,\sigma}(\A)$, we obtain a Hom-Lie algebra $(\mathcal{D}_{\tau,\sigma}(\A),[~,~],\alpha)$ with bracket defined as
\[
[d_n,d_m] = \frac{q^{-m}}{t^{-1} - qt} d_{n-m} - \frac{q^{-n}}{t^{-1} - qt} d_{m-n}
\]
and twisting map $\alpha = \overline{\sigma} \overline{\tau^{-1}} - id$ defined on generators by
\[
\alpha(d_n) = q^{-n} d_{-n} - d_n.
\]
The Hom-Jacobi identity writes
\[
\circlearrowleft_{n,m,k} [q^{-n} d_{-n} - d_n,[d_m,d_k]] = 0.
\]

\bigskip

\begin{remark}
To define the bracket of the \autoref{Thm:Bracket-Gen-Der} on $(\tau,\sigma)$-derivations, the map $\tau$ was needed to be invertible. We can define brackets on $\A \cdot \Delta$ without that assumption, but we obtain then restrictive conditions to get the properties wanted.

The following brackets are all obtained by inserting the map $\tau$ at various places in the original definition of the bracket for $\sigma$-derivations
\[
[a \cdot \Delta,b \cdot \Delta]_\sigma \coloneqq  (\sigma(a) \cdot \Delta) \circ (b \cdot \Delta) - (\sigma(b) \cdot \Delta) \circ (a \cdot \Delta)
\]
in order to recover this case when $\tau = id$.

\newcounter{case}
\setcounter{case}{1}

\begin{itemize}
\item $[a \cdot \Delta,b \cdot \Delta]_{\tau,\sigma}^{(0)} \coloneqq  (\sigma(a) \cdot \Delta) \circ (b \cdot \Delta) - (\sigma(b) \cdot \Delta)$. \\
It is a well-defined bracket if $\sigma(\Ann(\Delta)) \subset \Ann(\Delta)$.
\item $[a \cdot \Delta,b \cdot \Delta]_{\tau,\sigma}^{(s)} \coloneqq  (\sigma(a) \cdot \Delta) \circ (\tau(b) \cdot \Delta) - (\sigma(b) \cdot \Delta) \circ (\tau(a) \cdot \Delta)$. \\
It is a well-defined bracket if $\sigma(\Ann(\Delta)) \subset \Ann(\Delta)$ and $\tau(\Ann(\Delta)) \subset \Ann(\Delta)$.
\item $[a \cdot \Delta,b \cdot \Delta]_{\tau,\sigma}^{(\thecase)} \coloneqq  (\sigma(a) \cdot \tau \circ \Delta) \circ (b \cdot \Delta) - (\sigma(b) \cdot \tau \circ \Delta) \circ (a \cdot \Delta)$. \\
It is a well-defined bracket if $\sigma(\Ann(\Delta)) \subset \Ann(\Delta)$.
\addtocounter{case}{1}
\item $[a \cdot \Delta,b \cdot \Delta]_{\tau,\sigma}^{(\thecase)} \coloneqq  (\sigma(a) \cdot \Delta \circ \tau) \circ (b \cdot \Delta) - (\sigma(b) \cdot \Delta \circ \tau) \circ (a \cdot \Delta)$. \\
It is a well-defined bracket if $\sigma(\Ann(\Delta)) \subset \Ann(\Delta)$.
\addtocounter{case}{1}
\item $[a \cdot \Delta,b \cdot \Delta]_{\tau,\sigma}^{(\thecase)} \coloneqq  (\sigma(a) \cdot \Delta) \circ (b \cdot \tau \circ \Delta) - (\sigma(b) \cdot \Delta) \circ (a \cdot \tau \circ \Delta)$. \\
It is a well-defined bracket if $\sigma(\Ann(\Delta)) \subset \Ann(\Delta)$.
\addtocounter{case}{1}
\item $[a \cdot \Delta,b \cdot \Delta]_{\tau,\sigma}^{(\thecase)} \coloneqq  (\sigma(a) \cdot \Delta) \circ (b \cdot \Delta \circ \tau) - (\sigma(b) \cdot \Delta) \circ (a \cdot \Delta \circ \tau)$. \\
It is a well-defined bracket if $\sigma(\Ann(\Delta)) \subset \Ann(\Delta)$.
\end{itemize}

For the bracket labelled $(0)$ to be an internal law of composition on $\A \cdot \Delta$, the morphisms $\tau$ and $\sigma$ must satisfy
\begin{align*}
& \tau^2 = \tau \qquad \text{and} \qquad \tau \circ \sigma = \sigma \\
\text{or} \qquad & \tau^2 = \sigma \qquad \text{and} \qquad \tau \circ \sigma = \tau.
\end{align*}

In the case of the bracket labelled $(s)$ we need the previous condition and the following one :
\begin{equation} \label{eq-S-null}
\forall\, a,b \in \A, \qquad \sigma(a) \sigma(\tau(b)) - \sigma(b) \sigma(\tau(a)) = 0.
\end{equation}

In the cases $(1)$ and $(2)$, we need condition \eqref{eq-S-null} and also that $\tau^2=id$.

The bracket labelled $(3)$ only needs that $\tau^2=id$ to be an internal law of composition, whereas the bracket labelled $(4)$ needs that $\tau^3 = \tau$ and $\tau \circ \sigma \circ \tau = \sigma$, or $\tau \circ \sigma \circ \tau = \tau$ and $\tau^3 = \sigma$.
\end{remark}

\bibliographystyle{abbrv}
\nocite{*}
\bibliography{biblio-st-der}


\end{document}